\documentclass[11pt]{amsart}

\usepackage{amsmath,amssymb,amsthm,mathtools}
\usepackage{hyperref}
\hypersetup{%
  colorlinks,
  bookmarksopen,
  bookmarksnumbered,
  citecolor=blue,
  linkcolor=black,
  pdfstartview=FitH,
}

\usepackage{microtype}
\usepackage{pstricks}
\usepackage{auto-pst-pdf}
\usepackage{pst-node}
\usepackage{microtype}
\usepackage{color}
\newtheorem{lem}{Lemma}[section]

\newtheorem{prop}[lem]{Proposition}
\newtheorem{cor}[lem]{Corollary}
\newtheorem{thm}[lem]{Theorem}

\theoremstyle{definition}
\newtheorem{ques}[lem]{Question}
\newtheorem{defn}[lem]{Definition}
\newtheorem{notn}[lem]{Notation}
\newtheorem{rk}[lem]{Remark}
\numberwithin{equation}{section}
\let\Re\relax\DeclareMathOperator{\Re}{Re}
\let\Im\relax\DeclareMathOperator{\Im}{Im}

\DeclareMathOperator{\df}{tf}
\let\tf\df

\DeclareMathOperator{\conv}{conv}
\DeclareMathOperator{\trace}{trace}

\DeclareMathOperator{\diag}{diag}

\let\epsilon\varepsilon\let\phi\varphi

\newenvironment{smallbmatrix}{\left[\begin{smallmatrix}}{\end{smallmatrix}\right]}

{\let\privatefoo#1\privatefoo[\begin{smallmatrix}}%
  {\end{smallmatrix}\privatefoo]}

\let\sectmark\S
\def\dc#1{\expandafter\def\csname#1\endcsname{\mathcal{#1}}}
\def\db#1{\expandafter\def\csname b#1\endcsname{\mathbb{#1}}}

\def\loopy#1#2{%
  \def#1##1{\def\next{#2{##1}#1}\ifx##1\relax\let\next\relax\fi\next}}

\loopy{\makemathcals}{\dc}\loopy{\makemathbbs}{\db}

\makemathbbs  CRNQZTDFPE\relax
\makemathcals QWERTYUIOPASDFGHJKLZXCVBNM\relax

\makeatletter\newcount\@DT@modctr\newcount\@dtctr
\def\@modulo#1#2{\@DT@modctr=#1\relax\divide \@DT@modctr by #2\relax
\multiply \@DT@modctr by #2\relax\advance #1 by -\@DT@modctr}
\newcommand{\xxivtime}{\@dtctr=\time\divide\@dtctr by 60
\ifnum\@dtctr<10 0\fi\the\@dtctr:\@dtctr=\time\@modulo{\@dtctr}{60}%
\ifnum\@dtctr<10 0\fi\the\@dtctr}

\makeatother

\def\blob{\pscircle*{1.2pt}}%

\def\g#1#2#3{\,%
  \bgroup
  {
    \begin{postscript}\psset{nodesep=0pt,linewidth=.1ex}%
      \ensuremath{{\psmatrix[colsep=1.5ex,rowsep=-1.2ex,mnode=r]%
          \psset{nodesep=0pt,linewidth=.1ex}%
          #1\\#2%
          #3\endpsmatrix\,}}\end{postscript}}\egroup}

\def\gqqq{\g{\blob}{\blob}{\nn11}}
\def\gqe#1{\g{\,&\blob&\,}{\blob&\blob&\blob}{#1}}
\def\gqr#1{\g{\,&\blob&\,&\,}{\blob&\blob&\blob&\blob}{#1}}
\def\gww#1{\g{\blob&\blob}{\blob&\blob}{#1}}
\def\gwe#1{\g{\blob&\blob}{\blob&\blob&\blob}{#1}}
\def\gwr#1{\g{\blob&\blob}{\blob&\blob&\blob&\blob}{#1}}
\def\gee#1{\g{\blob&\blob&\blob}{\blob&\blob&\blob}{#1}}
\def\get#1{\g{\,&\blob&\blob&\blob}{\blob&\blob&\blob&\blob&\blob}{#1}}
\def\gte#1{\g{\blob&\blob&\blob&\blob&\blob}{\,&\blob&\blob&\blob}{#1}}
\def\ger#1{\g{\blob&\blob&\blob}{\blob&\blob&\blob&\blob}{#1}}
\def\grr#1{\g{\blob&\blob&\blob&\blob}{\blob&\blob&\blob&\blob}{#1}}
\def\gtr#1{\g{\blob&\blob&\blob&\blob&\blob}{\blob&\blob&\blob&\blob}{#1}}
\def\grt#1{\g{\blob&\blob&\blob&\blob}{\blob&\blob&\blob&\blob&\blob}{#1}}
\def\gre#1{\g{\blob&\blob&\blob&\blob}{\blob&\blob&\blob}{#1}}

\def\nn#1#2{\ncline{1,#1}{2,#2}}

\def\kqe{\gqe{\nn21\nn22\nn23}}

\def\gwwq{\gww{\nn11\nn12\nn22}}
\let\snakeww\gwwq
\def\vsnakewe{\nn11\nn12\nn22\nn23}
\def\gweq{\gwe{\vsnakewe}}
\let\snakewe\gweq

\def\geeq{\gee{\nn11\nn12\nn22\nn23\nn33}}
\let\snakeee\geeq
\def\veew{\nn11\nn21\nn22\nn23\nn33}
\def\geew{\gee{\veew}}
\def\geee{\gee{\nn11\nn12\nn21\nn22\nn23\nn32}}
\let\trie\geee
\def\geer{\gee{\nn11\nn12\nn21\nn22\nn23\nn32\nn33}}
\def\geet{\gee{\nn11\nn12\nn21\nn23\nn32\nn33}}
\let\loopee\geet

\def\snakeer{\ger{\nn11\nn12\nn22\nn23\nn33\nn34}}

\def\grew{\gre{\nn11\nn21\nn22\nn32\nn23\nn43}}

\def\looprr{\grr{\nn11\nn12\nn21\nn23\nn32\nn34\nn43\nn44}}

\def\snakerr{\grr{\nn11\nn12\nn22\nn23\nn33\nn34\nn44}}
\def\snakert{\grt{\nn11\nn12\nn22\nn23\nn33\nn34\nn44\nn45}}

\def\vrrz{\nn11\nn12\nn22\nn32\nn33\nn34\nn44}
\def\grrz{\grr{\vrrz}}
\def\vrrx{\veew\nn34\nn44}
\def\grrx{\grr{\vrrx}}
\def\grrc{\grr{\veew\nn24\nn44}}
\def\grrr{\grr{\nn11\nn21\nn22\nn32\nn23\nn43\nn44}}

\def\getq{\get{\nn21\nn22\nn32\nn33\nn34\nn44\nn45}}

\def\schur{\bullet}

\newcommand{\hsnorm}[1]{%
  \left\vert\kern-0.9pt\left\vert\kern-0.9pt\left\vert #1
      \right\vert\kern-0.9pt\right\vert\kern-0.9pt\right\vert}

\title{Norms of idempotent Schur multipliers}
\author{Rupert H. Levene}%
\email{rupert.levene@ucd.ie}%
\address{School of Mathematical Sciences\\University College Dublin\\Belfield\\Dublin~4\\Ireland}


\begin{document}

\begin{abstract}
  Let~$\D$ be a masa in~$\B(\H)$ where~$\H$ is a separable Hilbert
  space.  We find real %
  numbers~$\eta_0<\eta_1<\eta_2<\dots<\eta_6$ so that for every
  bounded, normal~$\D$-bimodule map~$\Phi$ on~$\B(\H)$,
  either~$\|\Phi\|>\eta_6$ or~$\|\Phi\|=\eta_k$ for some $k\in
  \{0,1,2,3,4,5,6\}$. When~$\D$ is totally atomic, these maps are the
  idempotent Schur multipliers and we characterise those with
  norm~$\eta_k$ for~$0\leq k\leq 6$. We also show that the Schur
  idempotents which keep only the diagonal and superdiagonal of
  an~$n\times n$ matrix, or of an~$n\times (n+1)$ matrix, both have
  norm~$\frac2{n+1}\cot(\frac\pi{2(n+1)})$, and we consider the
  average norm of a random idempotent Schur multiplier as a function
  of dimension. Many of our arguments are framed in the combinatorial
  language of bipartite
  graphs.  %
\end{abstract}
\maketitle
\tableofcontents

\section{Introduction}\label{sec:intro}

Let~$\bF$ be either $\bR$ or~$\bC$, and
let~$m,n\in\bN\cup\{\aleph_0\}$. If~$A=[a_{ij}]$ and~$X=[x_{ij}]$ are
$m\times n$ matrices with entries in~$\bF$, then the Schur product
of~$A$ and~$X$ is their entrywise product:
\[ A\schur X=[a_{ij}x_{ij}].\] This is also known as the Hadamard
product. Let~$\B=\B(\ell^2_n,\ell^2_m)$ be the space of matrices
defining bounded linear operators $\ell^2_n\to\ell^2_m$, where
$\ell^2_k$ is the $k$-dimensional Hilbert space of square-summable
$\bF$-valued sequences. An $m\times n$ matrix~$A$ with entries
in~$\bF$ is called a Schur multiplier if~$X\mapsto A\schur X$
leaves~$\B$ invariant. It then follows that Schur multiplication
by~$A$ defines a bounded linear map~$\B\to \B$, so the Schur norm
of~$A$ given by
\[ \|A\|_\schur=\sup\{ \|A\schur X\|_\B\colon {{X\in\B,\ \|X\|_\B\leq
    1}}\}\] is finite.  Under matrix addition, the Schur product
$\bullet$ and the norm $\|\cdot\|_\bullet$, the set of all $m\times n$
Schur multipliers forms a unital commutative semisimple Banach
algebra. Several properties of Schur multipliers and the
norm~$\|\cdot\|_{\bullet}$ are known; see for
example~\cite{bennett,mathias,dav-don}. Here, we focus on the norms of
the idempotent elements of this algebra: those Schur multipliers~$A$
for which every entry of~$A$ is either~$0$ or~$1$.

If~$S\subseteq \bF$, then we write $M_{m,n}(S)$ for the set of
all~$m\times n$ matrices with entries in~$S$.  For $m,n\in\bN$,
consider the finite sets of non-negative real numbers
\[ \N(m,n)=\{\|A\|_\schur \colon A\in M_{m,n}(\{0,1\})\}.\] We will
see in Remark~\ref{rk:real-complex} below that this set does not
depend on whether~$\bF=\bR$ or~$\bF=\bC$. Adding rows or columns of
zeros to a matrix does not change its Schur norm, so if $n\leq n'$ and
$m\leq m'$, then $\N(m,n)\subseteq \N(m',n')$. %
We will be interested in the set
\[ \N=\N(\aleph_0,\aleph_0)\] consisting of the norms of all
idempotent Schur multipliers on~$\B(\ell^2)$.
Every element of~$\N$ is the supremum of a sequence
in~$\bigcup_{m,n\in\bN}\N(m,n)$, obtained by considering the Schur
norms of the upper-left hand corners of the corresponding infinite
$0$--$1$ matrix.

It has been known for some time that~$\N$ is closed under
multiplication (consider $A_1\otimes A_2$) and under suprema (consider
$\bigoplus_{i} A_i$), that $\N$ is not bounded
above~\cite{kwapien-pel} and that~$\N$ contains accumulation
points~\cite{bcd}. %
On the other hand, many basic properties of~$\N$ seem to be
unknown. For example: is $\N$ closed?  Does $\N$ have non-empty
interior?  Might we have $\N \supseteq [a,\infty)$ for some $a\geq 0$?
Or, in the opposite direction, is~$\N$ actually countable?

We say that a non-empty open interval $(a,b)$ is a gap in~$\N$ if
$a,b\in \N$ but $(a,b)\cap \N=\emptyset$.  The idempotent elements~$p$
of any Banach algebra satisfy \[\|p\|=\|p^2\|\leq \|p\|^2,\] so
if~$\|p\|\leq 1$ then~$\|p\|\in \{0,1\}$.  In particular, this shows
that $(0,1)$ is a gap in~$\N$. However, $\N$ contains further gaps, a
perhaps unexpected phenomenon.
Indeed, Livschits~\cite{livschits} proves that
\[\{0,1,\sqrt{4/3}\}\subseteq \N\subseteq \{0,1\}\cup
[\sqrt{4/3},\infty),\] so the open interval~$(1,\sqrt{4/3})$ is also a
gap in~$\N$. Livschits' theorem has since been generalised by
Katavolos and Paulsen~\cite{kat-paulsen}, and has been recently used by
Forrest and Runde to describe certain ideals of the Fourier algebra of
a locally compact group~\cite{forrest-runde}.

We will show that there are at least four further
gaps:
\begin{thm}\label{thm:gaps}
  Consider the real numbers $\eta_0<\eta_1<\eta_2<\eta_3<\eta_4<\eta_5<\eta_6$ given by 
  \begin{gather*}
  \eta_0=0,\quad
  \eta_1=1,\quad
  \eta_2=\sqrt{\frac43},\quad
  \eta_3=\frac{1+\sqrt2}2,\\
  \eta_4=\frac1{15}\sqrt{169+38\sqrt{19}},\quad
  \eta_5=\sqrt{\frac 32},\quad
  \eta_6=\frac25\sqrt{5+2\sqrt5}.
\end{gather*}
We have
  \[
  \{\eta_0,\eta_1,\eta_2,\eta_3,\eta_4,\eta_5,\eta_6\}\subseteq \N\subseteq\{\eta_0,\eta_1,\eta_2,\eta_3,\eta_4,\eta_5\}\cup [\eta_6,\infty),\]
  so $(\eta_{j-1},\eta_{j})$ is a gap in~$\N$ for $1\leq j\leq 6$.
\end{thm}

Since it is fundamental to many of the calculations that follow, we
recall here the connection between the problem of finding
$\|A\|_{\schur}$ and factorisations $A=S^*R$. If~$m,n\in\bN$ and $A\in
M_{m,n}(\bC)$, the well-known Haagerup estimate (essentially attributed to
Grothendieck in~\cite{pisier-book}) states
\[ \|A\|_{\schur} \leq \|W\|\,\|V\|\quad\text{where}\quad A\schur
X=\sum_{j=1}^k W_jXV_j\text{ for all $X\in M_{m,n}(\bC)$}.\] Here $k$
is a natural number, $W$ is a block row of $m\times m$ matrices
$W_1,W_2,\dots,W_k$ and $V$ is a block column of $n\times n$ matrices
$V_1,V_2,\dots,V_k$; the norms of $V$ and $W$ are computed by allowing
them to act as linear operators between Hilbert spaces of the
appropriate finite dimensions. Moreover, the norm $\|A\|_{\schur}$ is
the minimum of these estimates $\|W\|\,\|V\|$. Stated in this
generality, the same is true for an arbitrary elementary operator
on~$M_{m,n}(\bC)$; for Schur multipliers, the minimum is attained by a
row $W$ and a column $V$ with $k\leq \min\{m,n\}$ for which the
entries of~$W$ and~$V$ are all \emph{diagonal} matrices. We can then
rewrite the Haagerup estimate in the compact form
\[ \|A\|_{\schur} \leq c(S) c(R) \quad \text{where}\quad A=S^*R\] by
taking $R$ to be the $k\times n$ matrix whose rows %
are the diagonals of the entries of~$V$, and $S$ to be the $k\times m$
matrix whose rows %
are the complex conjugates of the diagonals of the entries of~$W$, and
defining $c(R)$ and $c(S)$ to be the maximum of the $\ell^2$-norms of
the columns of the corresponding matrices~$R$ and~$S$. This notation
comes from~\cite{ang-cow-nar,cowen-et-al}.
\medskip

The structure of this paper is as follows. We will use the
combinatorial language of bipartite graphs to describe idempotent
Schur multipliers, and this is explained in
Section~\ref{sec:bipartite}. Section~\ref{sec:basicresults} briefly
recalls some basic results about the norms of general Schur
multipliers, and casts them in this language. Section~\ref{sec:snakes}
is concerned with the calculation of the norms of the idempotent Schur
multipliers corresponding to simple paths; these are the maps which
keep only the main diagonal and superdiagonal elements of a
matrix. Somewhat unexpectedly, we get the same answer in the $n\times
n$ and the $n\times (n+1)$ cases. In Section~\ref{sec:calcs} we
compute or estimate the norms of some ``small'' idempotent Schur
multipliers. Section~\ref{sec:proofmain} uses these results and simple
combinatorial arguments to characterise the Schur idempotents with
norm~$\eta_k$ for~$1\leq k\leq 6$, and hence to prove
Theorem~\ref{thm:gaps}. Using work of Katavolos and
Paulsen~\cite{kat-paulsen}, this allows us to show in
Section~\ref{sec:cbn} that these gaps persist in the set of norms of
all bounded, normal, idempotent masa bimodule maps on~$\B(\H)$
where~$\H$ is a separable Hilbert space.
Finally, in Section~\ref{sec:random} we estimate the average Schur
norm of a random Schur idempotent, in which each entry is chosen
independently to be~$1$ with probability~$p$ and~$0$ with
probability~$1-p$.

\section{Bipartite graphs}\label{sec:bipartite}

Let~$m,n\in\bN\cup \{\aleph_0\}$, and
consider an $m\times n$ matrix~$A=[a_{ij}]$ where each~$a_{ij}\in
\{0,1\}$. To~$A$ we associate an undirected countable bipartite
graph~$G=G(A)$, specified as follows. The vertex set~$V(G)$ is the
disjoint union of two sets, $R$ and~$C$, where $|R|=m$ and
$|C|=n$. Fixing enumerations $R=\{r_1,r_2,\dots\}$ and
$C=\{c_1,c_2,\dots\}$, we define the edge set of~$G$ to be
\[ E(G)=\big\{ (r_i,c_j) \colon a_{ij}=1\big\}.\] 
For example, if
\[ A=
\begin{bmatrix}
  1&1&0&0\\0&1&1&0\\0&0&0&1
\end{bmatrix},
\]
then the corresponding graph is
\[ G(A)=\ger{\nn11\nn12\nn22\nn23\nn34}\] where we have drawn the set
of ``row vertices'' $R=\{r_1,r_2,r_3\}$ above the ``column vertices''
$\{c_1,c_2,c_3,c_4\}$.  In general, $G$ will be a bipartite graph with
bipartition~$(R,C)$, which simply means that~$R\cap C=\emptyset$ and
every edge in~$G$ joins an element of~$R$ to an element of~$C$.  We
call such a graph an $(R,C)$-bipartite graph. Clearly the map
$A\mapsto G(A)$ is a bijection from the set of all $m\times n$
matrices of $0$s and~$1$s onto $\Gamma(R,C)$, the set of all
$(R,C)$-bipartite graphs. We remark in passing that in the linear
algebra and spectral graph theory literature, $A$ is called the
biadjacency matrix of~$G(A)$.

We will write $A=M(G)$ to mean that $G=G(A)$, and %
we adopt the shorthand
\[ \|G\|:=\|M(G)\|_{\schur}.\] In particular, if~$R$ and~$C$ are
countably infinite sets, then \[\N=\{\|G\|\colon G\in \Gamma(R,C)\}\setminus\{\infty\}.\]

More generally, if~$X$ and~$Y$ are any sets and~$G\subseteq X\times
Y$, then we may think of~$G$ as a bipartite graph whose vertex
set~$V(G)$ is the disjoint union of~$X$ and~$Y$, and whose edge set
is~$E(G)=G$. We write $\Gamma(X,Y)$ for the power set of~$X\times Y$,
viewed as the collection of all such bipartite graphs.

If~$G\in \Gamma(X,Y)$ and $G'\in \Gamma(X',Y')$, then we say that the
graphs $G$ and $G'$ are isomorphic if there is an isomorphism of
bipartite graphs from~$G$ to~$G'$. This means that there is a
bijection $\theta\colon V(G)\to V(G')$ which either maps $X$ onto $X'$
and $Y$ onto~$Y'$ or maps $X$ onto $Y'$ and $Y$ onto $X'$, so
that~$\theta$ induces a bijection from~$E(G)$ onto~$E(G')$.
We do not distinguish between isomorphic graphs, so for example we
write $G=G'$ if $G$ and $G'$ are merely isomorphic.

If $G_0\in \Gamma(X_0,Y_0)$ and~$G\in\Gamma(X,Y)$, then $G_0$ is an
induced subgraph of~$G$ if $X_0\subseteq X$, $Y_0\subseteq Y$ and for
$x_0\in X_0$ and $y_0\in Y_0$ we have
\[ (x_0,y_0)\in E(G_0)\iff (x_0,y_0)\in E(G).\] In other
words,~$G_0=G\cap (X_0\times Y_0)$; we will abbreviate this as
$G_0=G[X_0,Y_0]$.

If we merely have
\[ (x_0,y_0)\in E(G_0)\implies (x_0,y_0)\in E(G),\] so that $G_0$ may
be obtained by removing some edges from an induced subgraph of~$G$,
then we say that $G_0$ is a subgraph of~$G$.  We will write
$G_0\leq G$ or $G\geq G_0$ to mean that $G_0$ (or a graph isomorphic
to~$G_0$) is an induced subgraph of~$G$; and we will write
$G_0\subseteq G$ to mean that $G_0$ (or a graph isomorphic to~$G_0$)
is a subgraph of~$G$. Similarly, we write $G_0< G$ to mean
that~$G_0\leq G$ but~$G_0$ is not isomorphic to~$G$.

\begin{defn}
  Let~$G$ be a graph and let~$v$ be a vertex of~$G$. The set~$N(v)$ of
  neighbours of~$v$ in~$G$ consists of all vertices joined to~$v$ by
  an edge of~$G$.  The degree~$\deg(v)$ of~$v$ in~$G$ is the
  cardinality of~$N(v)$. If the vertices of~$G$ have bounded degree,
  then we write
  \[\deg(G)=\max_{v\in V(G)}\deg(v),\]
  and we write $\deg(G)=\infty$ otherwise.
  
  We say that vertices~$v,w$ in~$G$ are
  twins in~$G$ if $N(v)=N(w)$. A graph~$G$ is
  twin-free if no pair of distinct vertices are
  twins.
\end{defn}

\begin{prop}\label{prop:dupe-free}
  Any graph~$G$ has a maximal twin-free induced
  subgraph~$\df(G)$, which is unique up to graph isomorphism.
  If~$G$ is bipartite, then so is~$\df(G)$.
\end{prop}
\begin{proof}
  Being twins is an equivalence relation on the vertices
  of~$G$. If we choose a complete set of equivalence class
  representatives, then the corresponding induced subgraph of~$G$ is
  twin-free, and by construction it is maximal with respect
  to~$\leq$ among the twin-free induced subgraphs of~$G$. Passing from
  one choice of equivalence class representatives to another produces
  an isomorphism of graphs. On the other hand, if~$v$ and~$w$ are any
  two distinct vertices in a twin-free induced subgraph $S\leq
  G$, then $v$ and $w$ are not twins in~$S$, so they cannot be
  twins in~$G$. So the vertices of~$S$ all lie in different
  equivalence classes, so $S$ is an induced subgraph of one of the
  maximal induced subgraphs we have described. Since any subgraph of a
  bipartite graph is bipartite, the second assertion is trivial.
\end{proof}

\begin{rk}
  Note that $M(\tf(G))$ is obtained from $M(G)$ by repeatedly deleting
  duplicate rows and columns.
\end{rk}

Let~$G$ be any graph. If~$v,v'$ are distinct vertices of~$G$, then a
path in~$G$ from~$v$ to $v'$ of length~$k$ is a finite sequence
$(v_0,v_1,v_2,\dots,v_k)$ of vertices of~$G$, where~$v=v_0$ and
$v'=v_k$, so that~$v_j$ is joined by an edge in~$G$ to~$v_{j+1}$ for
$1\leq j<k$. This is a simple path if no vertex appears twice. The
distance between~$v$ and~$v'$ is the smallest possible length of such
a path in~$G$. Being joined by some path in~$G$ is an equivalence
relation on the vertices of~$G$; by a connected component of~$G$ we
mean an equivalence class for this relation, and we say that~$G$ is
connected if it is a connected component of itself.

It is easy to see that:
\begin{lem}\label{lem:df-conn}
  A graph~$G$ is connected if and only if~$\df(G)$ is connected.\qed
\end{lem}

The size $|G|$ of a graph~$G$ is the cardinality of its vertex set. We
say that~$G$ is finite if~$|G|<\infty$. Let~$\F(G)$ be the set
\[ \F(G) = \{F\leq G\colon \text{$F$ is finite, connected and
  twin-free}\}.\]
We will use the following
observation in Section~\ref{sec:cbn}.

\begin{lem}\label{lem:finite-connected-subgraphs}
  Let~$X,Y$ be sets and let~$G\in \Gamma(X,Y)$ be a connected
  bipartite graph. 
  If~$\F(G)$ contains finitely many non-isomorphic bipartite graphs,
  then~$\df(G)$ is finite.
\end{lem}
\begin{proof}
  Suppose instead that $\df(G)=G[S,T]$ where~$S\subseteq X$
  and~$T\subseteq Y$ and~$S$ is infinite.  Let~$A$ be a finite subset
  of~$S$ with~$|A|>|F|$ for every~$F\in \F(G)$. Since~$G[S,T]$ is
  twin-free, for any pair~$a_1,a_2$ of distinct vertices in~$A$ there
  is a vertex~$t=t(a_1,a_2)\in T$ so that one of~$(a_1,t)$
  and~$(a_2,t)$ is an edge of~$G$, and the other is
  not. Consider \[B=\{t(a_1,a_2)\colon a_1,a_2\in A,\ a_1\ne a_2\}.\]
  Since~$\df(G)$ is connected by Lemma~\ref{lem:df-conn}, we can find
  finite sets~$A',B'$ with $A\subseteq A'\subseteq S$ and~$B\subseteq
  B'\subseteq T$ so that $G[A',B']$ is connected. Consider
  $F=\df(G[A',B'])$. By construction, $F\in \F(G)$. However, $|F|\geq
  |A|$ since no two vertices in~$A$ are twins in~$G[A',B']$, so~$F$
  cannot be (isomorphic to) an element of~$\F(G)$, a
  contradiction.%
\end{proof}

\section{Basic results}\label{sec:basicresults}

If~$A$ and~$B$ are matrices, then we will
write
\[ A\simeq B\] to mean that $B=UAV$ for some permutation
matrices~$U,V$; in other words, permuting the rows and columns of~$A$
yields $B$.

The following facts about the norms of Schur multipliers are
well-known.
\begin{prop}\label{prop:basic-matrices}
  Let~$A$ and~$B$ be matrices with countably many rows and columns.
  \begin{enumerate}
  \item $\|A\|_\schur=\|A^t\|_{\schur}$
  \item If~$A\simeq B$, then $\|A\|_\schur=\|B\|_\schur$.
  \item If~$B$ can be obtained from~$A$ by deleting some rows
    or columns, then $\|B\|_{\schur}\leq \|A\|_{\schur}$.
  \item $\|A_1\oplus A_2\oplus A_3\oplus\dots\|_\schur=\sup_{j}\|A_j\|_{\schur}$
  \item $\|A\|_{\schur}=\left\|\begin{smallbmatrix}
          A&A\\A&A
        \end{smallbmatrix}\right\|_{\schur}$
    \item If~$B$ can be obtained from~$A$ by duplicating rows or
      columns, then $\|B\|_{\schur}=\| A\|_{\schur}$.
  \end{enumerate}
\end{prop}
\begin{proof}
  Statements (1)--(4) all follow easily from properties of the
  operator norm~$\|\cdot\|_\B$. %
  For (5), let us write $S_A\colon \B\to\B$, $X\mapsto A\schur X$ for
  the mapping of Schur multiplication by~$A$.  The two-fold ampliation
  $S_A^{(2)}\colon M_2(\B)\to M_2(\B)$ of~$S_A$ (in the sense of
  operator space theory) may be naturally identified with $S_B$ where
  $B=\begin{smallbmatrix}
          A&A\\A&A
        \end{smallbmatrix}$. 
      Now \[\|A\|_{\schur}=\|S_A\|\leq
      \|S_A^{(2)}\|=\|B\|_{\schur}\leq
      \|S_A\|_{cb}=\|S_A\|=\|A\|_\schur\] where the equality
      $\|S_A\|_{cb}=\|S_A\|$ is a theorem commonly attributed to an
      unpublished manuscript of Haagerup 
      (see~\cite[p.~115]{paulsen-book}, for example) and is also
      established in~\cite{smith91}. We therefore have equality,
      hence~(5).  Replacing the number~$2$ in this argument with some
      other countable cardinal and using statement~(3) then yields a
      proof of statement (6).
\end{proof}

Specialising to idempotent Schur multipliers and restating in terms of
bipartite graphs, we have: %
\begin{prop}\label{prop:basic-graphs}
  Let~$R$ and~$C$ be countable sets, and let $G\in \Gamma(R,C)$.
  \begin{enumerate}
  \item If~$G'\in\Gamma(R',C')$ and $G'$ is isomorphic to~$G$, then
    $\|G'\|=\|G\|$.
  \item If~$G_0\leq G$, then $\|G_0\|\leq \|G\|$.
  \item The norm of $G$ is the supremum of the norms of the
    connected components of~$G$.
  \item $\|G\|=\|\df(G)\|$.
  \end{enumerate}
\end{prop}
\begin{proof}
  (1) follows from assertions~(1) and~(2) of
  Proposition~\ref{prop:basic-matrices}.
  For $j=2,3$, assertion~($j$) here is a rewording of
  assertion~($j+1$) of~Proposition~\ref{prop:basic-matrices}.
  (4) follows easily using %
  the proof of Proposition~\ref{prop:dupe-free} and
  Proposition~\ref{prop:basic-matrices}(6).
\end{proof}

\begin{rk}
  It is natural to ask whether Proposition~\ref{prop:basic-graphs}(2)
  generalises to to all subgraphs, and not merely induced
  subgraphs. In other words, is following implication valid?
  \[ G_0 \subseteq G\stackrel{\text{?}}{\implies} \|G_0\|\leq \|G\|\]
  The answer is no. The complete graph $K$
  in~$\Gamma(\aleph_0,\aleph_0)$ corresponds to the matrix of all
  $1$s, which has Schur multiplier norm~$1$ since it gives the
  identity mapping, but as is well-known~\cite{kwapien-pel}, the
  upper-triangular subgraph $T\subseteq K$ whose matrix is
  \[ M(T)=
  \begin{bmatrix}
    1&1&1&\dots\\
    0&1&1&\dots\\
    0&0&1&\dots\\
    \vdots &\vdots &\vdots &\ddots
  \end{bmatrix}\] has $\|T\|=\infty$. Note that~$T$ is twin-free,
  but~$K$ is certainly not. In view of
  Proposition~\ref{prop:basic-graphs}(4), we might then ask whether
  this implication holds provided either~$G$ alone, or both~$G$
  and~$G_0$, are required to be twin-free. Again, the answer is
  no; a counterexample is given by~(7) and~(8) of
  Proposition~\ref{prop:smallnorms} below.
\end{rk}

\begin{rk}\label{rk:real-complex}
  We now explain why our results are identical regardless of whether
  we choose~$\bF=\bR$ or~$\bF=\bC$; we are grateful to an anonymous
  referee for providing the following simple argument. Let~$\B_\bF$ be
  the space of bounded linear maps from $\ell^2_{n,\bF}$ to
  $\ell^2_{m,\bF}$, the corresponding $\ell^2$ spaces with entries
  in~$\bF$. For a Schur multiplier~$A\in M_{m,n}(\bF)$, we temporarily
  write~$\|A\|_{\schur,\bF}$ for the norm of the map $\B_{\bF}\to
  \B_{\bF}$, $B\mapsto A\schur B$. %
  Given $X\in M_{m,n}(\bC)$, $\alpha\in\bC^m$ and~$\beta\in\bC^n$,
  write $\alpha_i=|\alpha_i|v_i$ and $\beta_j=|\beta_j|w_j$
  where~$v_i,w_j\in\bT$, and let~$\widetilde x_{ij} =
  \Re(x_{ij}v_iw_j)$. We have $\Re(x_{ij}\alpha_i\beta_j)=\widetilde
  x_{ij}|\alpha_i|\,|\beta_j|$, and the matrix~$\widetilde
  X=[\widetilde x_{ij}]\in M_{m,n}(\bR)$ has norm~$\|\widetilde
  X\|_{M_{m,n}(\bR)}\leq \|X\|_{M_{m,n}(\bC)}$. So
  \begin{align*}
  \|A\|_{\schur,\bC} &= \sup\left\{\left|\sum_{i,j} a_{ij}x_{ij}\alpha_i\beta_j\right|\colon \|X\|_{M_{m,n}(\bC)},\,\|\alpha\|_{\bC^m},\,\|\beta\|_{\bC^n}\leq 1\right\}
  \\
  &= \sup\left\{\sum_{i,j} a_{ij}\Re(x_{ij}\alpha_i\beta_j)\colon \|X\|_{M_{m,n}(\bC)},\,\|\alpha\|_{\bC^m},\,\|\beta\|_{\bC^n}\leq 1\right\}
  \\
  &\leq \sup\left\{\sum_{i,j} a_{ij}\widetilde x_{ij}\alpha_i\beta_j\colon \|\widetilde X\|_{M_{m,n}(\bR)},\,\|\alpha\|_{\bR^m},\,\|\beta\|_{\bR^n}\leq 1\right\}
  \\
  &=\|A\|_{\schur,\bR}.
  \end{align*}
  Since the reverse inequality is trivial, we have equality.

  For~$\bF=\bC$, this allows a slight simplification of the formula
  defining the norm of a Schur multiplier with real entries.  Indeed,
  any Schur multiplier~$A\in M_{m,n}(\bR)$ has
  \[\|A\|_{\schur}=\sup_{X\in O(m,n)} \|A\schur X\|_{\B_{\bR}}\]
  where~$O(m,n)$ is the set of extreme points of the unit ball
  of~$\B_{\bR}$: the set of isometries in~$\B_\bR$ if $n\leq m$, or
  the set of coisometries if~$n\geq m$.  %
\end{rk}

\section{Norms of simple paths}\label{sec:snakes}

\begin{defn}
  For~$n\in\bN$, the $n$-cycle~$\Lambda(n)$ is the maximal cycle
  in~$\Gamma(n,n)$; equivalently, it is connected and every vertex has
  degree~$2$. For example, $\Lambda(3)=\loopee$.
  
  The $(n,n)$ path~$\Sigma(n,n)$ and the $(n,n+1)$
  path~$\Sigma(n,n+1)$ are the maximal simple paths in~$\Gamma(n,n)$
  and~$\Gamma(n,n+1)$, respectively; for example,
  $\Sigma(3,3)=\snakeee$ and $\Sigma(3,4)=\snakeer$. %
\end{defn}

For~$n\in\bN$, we write $\theta_n=\frac{\pi}{2n}$.
By~\cite[Example~4.7]{dav-don}, we have
\begin{equation}%
  \label{eq:davdon-loops} \|\Lambda(n)\|=
  \begin{cases}
  \frac{2}{n}\cot\theta_{n}&\text{if~$n$ is even,}\\
  \frac{2}{n}\csc\theta_{n}&\text{if~$n$ is odd.}
\end{cases}
\end{equation}
The main result of this section is:
\begin{thm}\label{thm:snakes}
  For every~$n\in\bN$,  \[\|\Sigma(n,n)\|=\|\Sigma(n,n+1)\|=\frac
  2{n+1}\cot \theta_{n+1}.\]
\end{thm}
Before the proof, we make some remarks.

\begin{rk}
  Observe that while $\Sigma(n,n)< \Sigma(n,n+1)<\Lambda(n+1)$,
  the norms of the first two graphs are equal for every~$n$ and all
  three have equal norm for odd~$n$. However, these assertions do
  not follow from any of the easy observations of
  Proposition~\ref{prop:basic-graphs} since these graphs are all
  connected and twin-free.
\end{rk}
\begin{ques}
  Is there a combinatorial characterisation of the connected twin-free
  bipartite graphs~$G_0< G$ with $\|G_0\|=\|G\|$?
\end{ques}

\begin{rk}
  Theorem~\ref{thm:snakes} improves the following bounds of
  Popa~\cite{popa-thesis}:
  \[\frac1n\left(\csc\left(\frac{\pi}{4n+2}\right)-1\right)\leq
  \|\Sigma(n,n)\|\leq \frac2{n+1}\cot \theta_{n+1}.\] She establishes
  the upper bound using results of Mathias~\cite{mathias}, and the
  lower bound using some eigenvalue formulae due to Yueh~\cite{yueh}.
\end{rk}

The following corollary is also noted in~\cite{popa-thesis}.  Another
proof can be found by applying a theorem of
Bennett~\cite[Theorem~8.1]{bennett} asserting that the norm of a
Toeplitz Schur multiplier~$A$ is the total variation of the Borel
measure~$\mu$ on~$\bT$ with~$a_{i-j}=\hat
\mu(i-j)$.%

\begin{cor}\label{cor:4/pi}
  The infinite matrix $A$ with $a_{ij}=1$ if $j\in \{i,i+1\}$ and
  $a_{ij}=0$ otherwise has $\|A\|_{\schur}=4/\pi$.
\end{cor}
\begin{proof}
  The Schur multiplier norm of~$A$ is the supremum of the Schur
  multiplier norms of its $n\times n$ upper left-hand corners~$A_n$,
  and $G(A_n)=\Sigma(n,n)$. Hence
  \[ 
  \|A\|_{\schur}= \sup_{n\ge1}\|A_n\|_{\schur}=\sup_{n\ge1}
  \frac{2}{n+1}\cot\theta_{n+1}=\frac
  4\pi.\qedhere\]
\end{proof}

Recall that~$\N$ denotes the set of norms of all (bounded) Schur
idempotents.

\begin{rk}
  $\N$ is not discrete: its accumulation points include $2$
  by~\cite{bcd}, $\sqrt 2$ by Remark~\ref{rk:sqrt2} below, and $4/\pi$
  by Corollary~\ref{cor:4/pi}. By Theorem~\ref{thm:gaps}, the infimum
  of the set of accumulation points of~$\N$ is in the
  interval~$[\eta_6,4/\pi]$.
\end{rk}

\begin{ques}
  Is this infimum equal to~$4/\pi$?
\end{ques}

\begin{ques}
  Is~$\N$ closed? Does it have non-empty interior? Are there any limit
  points from above which are not limit points from below?
\end{ques}

We turn now to the proof of Theorem~\ref{thm:snakes}, which will
occupy us for the rest of this section.  Fix~$n\in\bN$. For~$j\in\bZ$, write
\[\kappa(j)=\cos(j\theta_{n+1})\quad\text{and}\quad \lambda(j)=\sin(j\theta_{n+1})\]
where as above, $\theta_{n+1}=\tfrac\pi{2(n+1)}$.
 Clearly,~$\lambda(j)=0\iff j\in 2(n+1)\bZ$.  The
following useful identity, valid for~$N\in\bN$, $f\in
\{\kappa,\lambda\}$ and $a,d\in\bZ$ with $\lambda(d)\ne0$, is an
immediate consequence of the formulae in~\cite{knapp}.
\begin{equation}%
  \label{eq:bigstar}
  \sum_{j=0}^{N} f(a+2dj) = \frac{\lambda((N+1)d)}{\lambda(d)} f(a+Nd)
\end{equation}

\begin{lem}\label{lem:altzero}
  Let $a\in\bZ$ and let~$f,g,h\in \{\pm \kappa,\pm\lambda\}$. 
  \begin{enumerate}
  \item If $m\in2\bZ$ and~$|m|\leq 2n$, then
    \[\displaystyle\sum_{j=0}^{2n+1} (-1)^j f(a+mj)=0.\]
  \item If~$s,t\in\bZ$ with $\max\{|s|,|t|\}\leq
    n-1$ and~$s\equiv t\pmod 2$, then
    \[\displaystyle \sum_{j=0}^{2n+1}(-1)^j f(a+2j)g(sj)h(tj)=0=\sum_{j=0}^{2n+1}(-1)^j g(sj)h(tj).\]
  \end{enumerate}
\end{lem}
\begin{proof}
  \begin{enumerate}
  \item We have
    \begin{align*}
      \sum_{j=0}^{2n+1} (-1)^j f(a+mj)
      =
      \sum_{j=0}^n f(a+2mj) - \sum_{j=0}^n f(a+m + 2mj).
    \end{align*}
    If~$m=0$ then this difference is clearly~$0$. If~$m\ne 0$, then
    $|m|<2(n+1)$ gives $\lambda(m)\ne0$ and $\lambda((n+1)m)=0$
    since~$m$ is even, so by equation ~\eqref{eq:bigstar} we have \[\sum_{j=0}^n
    f(a+2mj)=\sum_{j=0}^n f(a+m + 2mj)=0. \]
  \item Using the product-to-sum trigonometric identities, we can write
    \[ f(a+2j)g(sj)h(tj) = \frac14\sum_{k=0}^3f_k(a+m_kj)\]
    where~$f_k\in \{\pm \kappa,\pm \lambda\}$ and \[m_k = 2+(-1)^ks +
    (-1)^{\lfloor k/2\rfloor} t.\] Since~$m_k$ is even and~$|m_k|\leq
    2+ |s|+|t|\leq 2n$, the first equality follows from~(1). The
    second equality is proven using a simplification of the same
    argument.\qedhere
  \end{enumerate}
\end{proof}

Let~$\rho$ be the $2\times 2$ rotation matrix\[\rho=
\begin{bmatrix}
  \kappa(1)&-\lambda(1)\\\lambda(1)&\kappa(1)
\end{bmatrix}.\]Note that for~$s\in\bZ$, we have
\[ \rho^s =
\begin{bmatrix}
  \kappa(s)&-\lambda(s)\\\lambda(s)&\kappa(s)
\end{bmatrix}\] so that, in particular, each entry of~$\rho^s$ is
of the form~$g(s)$ for some~$g\in \{\kappa,\pm\lambda\}$.  Define
an~$n\times n$ orthogonal matrix~$W$ by
\[ W=
\begin{cases}
\rho\oplus\rho ^{3}\oplus \rho^5\oplus\dots\oplus\rho^{n-1 }  &\text{if~$n$ is even,}\\
[1]\oplus \rho^2\oplus\rho ^{4}\oplus\dots\oplus\rho^{n-1}  &\text{if~$n$ is odd.}
\end{cases}\] Here, $[1]$ is the $1\times 1$ matrix whose entry
is~$1$ and $\oplus$ is the block-diagonal direct sum of matrices.
Let $v$ be the $n\times 1$ vector
\[
v=
\begin{cases}
  [1\ 0\ 1\ 0\ \dots\ 1\ 0]^*&\text{if~$n$ is even,}\\
  [1\ 1\ 0\ 1\ 0\ \dots\ 1\ 0]^*&\text{if~$n$ is odd.}
\end{cases}
\]
For $j\in\bZ$, let $q_j=W^j v$, and
consider the rank one operators \[ Q_j = q_jq_j^*.\]
We write $\conv(S)$ for the convex hull of a subset~$S$ of a vector space.
\begin{prop}\label{prop:conv}
  Consider the real numbers~$t_j=\kappa(1)-\kappa(3+4j)$ for~$0\leq j\leq n$.
  \begin{enumerate}
  \item $t_j>0$ for~$0\leq j< n$ and $t_n=0$.\\[-3pt]
  \item $\displaystyle\sum_{j=0}^{2n+1}(-1)^j \kappa(a+2j)Q_j=0=\displaystyle\sum_{j=0}^{2n+1}(-1)^j Q_j$
    for any~$a\in\bR$.\\[6pt]
  \item
    $\displaystyle\sum_{j=0}^{n-1}t_jQ_{2j}=\sum_{j=0}^{n-1}t_jQ_{2(n-j)-1}$.\\[-3pt]
  \item $\conv(\{Q_0,Q_2,\dots,Q_{2(n-1)}\})\cap
    \conv(\{Q_1,Q_3,\dots,Q_{2n-1}\})\ne\emptyset$.
  \end{enumerate}
\end{prop}
\begin{proof}
  \begin{enumerate}
  \item This is clear.
  \item The~$k$th entry of~$q_j$ has the form $g_k(s_kj)$
    where~$g_k\in \{\kappa,\pm\lambda\}$ and~$s_k\in\bZ$ with
    $|s_k|\leq n-1$ and $s_k\equiv n-1\pmod2$. Hence the~$(k,\ell)$
    entry of~$Q_j$ is $g_k(s_kj)g_\ell(s_\ell j)$, so the claim
    follows from Lemma~\ref{lem:altzero}(2).
  \item For~$\ell\in\bZ$, we have~$W^\ell Q_j W^{-\ell} =
    Q_{j+\ell}$. By~(2),
    \[ W^{-1} \left(\sum_{j=0}^{2n+1}(-1)^j \kappa(1+2j)Q_j\right)W =
    \sum_{j=-1}^{2n} (-1)^{j+1}\kappa(3+2j) Q_j=0.\] Rearranging,
    reindexing and using the identity $\kappa(4(n+1)-x)=\kappa(x)$ gives    
    \begin{align*}
      \sum_{j=0}^n \kappa(3+4j)Q_{2j} = \sum_{j=0}^n \kappa(3+4j)Q_{2(n-j)-1}.
    \end{align*}
    We have~$W^{2(n+1)}=(-1)^{n+1}I$, so
    \[Q_{-1}=W^{2(n+1)}Q_{-1}W^{-2(n+1)}=Q_{2n+1}.\] By the second
    equality in~(2),
    \[ \sum_{j=0}^n \kappa(1)Q_{2j} = \sum_{j=0}^n \kappa(1)Q_{2(n-j)-1}.\]
    Taking differences gives
    \[ \sum_{j=0}^n t_j Q_{2j} = \sum_{j=0}^n t_j Q_{2(n-j)-1}.\]
    Since~$t_n=0$, this is the desired identity.
  \item This is immediate from~(1) and~(3).\qedhere
  \end{enumerate}
\end{proof}

For~$1\leq j\leq n-1$, let
\[ r_j =
\begin{cases}
  \sqrt{\tfrac2{n+1}}&\text{if~$j=1$ and~$n$ is odd},\\
  \sqrt{\tfrac{4\kappa(j)}{n+1}}&\text{otherwise}.
\end{cases}\]
Let $r$ be the $n\times 1$ vector
\[
r=
\begin{cases}
  [r_1\ 0\ r_3\ 0\ \dots\ r_{n-1}\ 0]^*&\text{if~$n$ is even,}\\
  [r_1\ r_2\ 0\ r_4\ 0\ \dots\ r_{n-1}\ 0]^*&\text{if~$n$ is odd.}
\end{cases}
\]
A calculation using equation~\eqref{eq:bigstar} gives
\[ \|r\|^2= \frac2{n+1}\cot \theta_{n+1}.\] 
Consider the rank one operators
\[ P_j=W^jr(W^jr)^*\] for~$j\in\bZ$. 

\begin{rk}\label{rk:PQ}
  The diagonal matrix
  \[D=
  \begin{cases}
    \diag(r_1,r_1,r_3,r_3,\dots,r_{n-1},r_{n-1})&\text{if $n$ is even}\\
    \diag(r_1,r_2,r_2,r_4,r_4,\dots,r_{n-1},r_{n-1})&\text{if $n$ is odd}
  \end{cases}
\] commutes with~$W$
  and $Dv=r$; hence $DQ_jD=P_j$. Since~$D$ is invertible, it follows
  that for any finite collection of scalars~$t_j$ we have
  \[ \sum_j t_jP_j=0\iff \sum_j t_j Q_j=0.\]
\end{rk}

Let~$R=[r\ W^2r\ W^4r\ \dots\ W^{2(n-1)}r]$ and let~$S=WR$.  Also
let~$\tilde R=[R\ W^{2n}r]$ and let~$\tilde S=W\tilde R$. Let us
write~$X_i$ for the $i$th column of a matrix~$X$. Since~$W$ is an
isometry, for $1\leq i,j\leq n+1$ we have
\[ \|\tilde R_j\|^2=\|\tilde S_i\|^2=\|r\|^2%
= \frac2{n+1}\cot\theta_{n+1}.\] Let
$\tilde B$ be the $(n+1)\times (n+1)$ matrix whose $(i,j)$ entry is
  \[b_{ij}=
  \begin{cases}
    1&\text{if~$j-i\in \{0,1\}$,}\\
    (-1)^{n+1}&\text{if~$(i,j)=(n+1,1)$,}\\
    0&\text{otherwise}.
  \end{cases}
  \]
  Let~$B$ be the upper-left $n\times n$ corner of~$\tilde B$ and
  let~$B'$ consist of the first~$n$ rows of~$\tilde B$. Observe that
  $G(B)=\Sigma(n,n)$, $G(B')=\Sigma(n,n+1)$ and if~$n$ is odd,
  then~$G(\tilde B)=\Lambda(n+1)$.

\begin{prop}\label{prop:upperbound}
  We have $S^*R=B$ and $\tilde S^*\tilde R=\tilde B$.
\end{prop}
\begin{proof}
  Since~$S^*R$ is the upper-left $n\times n$ corner of~$\tilde
  S^*\tilde R$, it suffices to show that~$\tilde S^* \tilde R=\tilde
  B$. Let~$k=|2(i-j)+1|$, a positive odd integer. We have
  \[(\tilde S^* \tilde R)_{i,j} = \langle \tilde R_j,\tilde S_i\rangle
  = \langle W^{2(j-1)}r,W^{2(i-1)+1}r\rangle=\langle
  W^{k}r,r\rangle.\] Since~$W^{2(n+1)}=(-1)^{n+1}I$, the $(n+1,1)$
  entry of~$\tilde S^*\tilde R$ is
  \[ \langle W^{2n+1}r,r\rangle = (-1)^{n+1}\langle W^*r,r\rangle =
  (-1)^{n+1}\langle Wr,r\rangle.\] It therefore only remains to show that
  \[
  \langle W^k r,r\rangle =
  \begin{cases}
    1&\text{if }k=1\\
    0&\text{if $k$ is odd with $3\leq k\leq 2n-1$}.
  \end{cases}
  \]
  We prove this by direct calculation, giving the details for
  even~$n$; the calculation for odd~$n$ is very similar.  We have
  \begin{align*}
    \langle Wr,r\rangle &= \sum_{j=0}^{n/2-1} r_{1+2j}^2 \kappa(1+2j) 
    \\
    &= \frac4{n+1} \sum_{j=0}^{n/2-1} \kappa(1+2j)^2\\
    &= \frac2{n+1} \sum_{j=0}^{n/2-1} 1+\kappa(2+4j)\\
    &=\frac2{n+1}\left(\frac n2 + \frac12\right) = 1.
  \end{align*}
  Here we have used equation~\eqref{eq:bigstar} to perform the summation in the
  penultimate line.  If~$3\leq k\leq 2n-1$ and~$k$ is odd, then
  \begin{align*}
    \langle W^kr,r\rangle &=
    \sum_{j=0}^{n/2-1} r_{1+2j}^2\, \kappa(k(1+2j))\\
    &=
    \frac 4{n+1} \sum_{j=0}^{n/2-1} \kappa(1+2j)\,\kappa(k(1+2j))\\
    &=
    \frac 2{n+1} \sum_{j=0}^{n/2-1} \kappa((k-1)(1+2j))+\kappa((k+1)(1+2j))\\
    &=\frac1{n+1}\left(
    \frac{ \lambda(n(k-1))}{\lambda(k-1)} + \frac{\lambda(n(k+1))}{\lambda(k+1)}\right).
  \end{align*}
  Since~$k$ is odd, $(k\pm1)(n+1)\theta_{n+1}=\frac{k\pm1}2 \pi$ is an
  integer multiple of~$\pi$ and
  \[\lambda(n(k\pm1))=\lambda((k\pm1)(n+1)-(k\pm1))=(-1)^{(k\mp1)/2}\lambda(k\pm1),\]
  so
  \[\langle W^kr,r\rangle=\frac{1}{n+1}\left( (-1)^{(k+1)/2} +
    (-1)^{(k-1)/2}\right)=0.\qedhere\]
\end{proof}

\begin{proof}[Proof of Theorem~\ref{thm:snakes}]
  By Proposition~\ref{prop:conv} and Remark~\ref{rk:PQ}, there are two
  sets of positive scalars~$\{a_j\}_{j=1}^n$ and~$\{b_j\}_{j=1}^n$,
  each summing to~$1$, so that
  \[\sum_{j=1}^n a_j R_jR_j^* = \sum_{j=1}^n b_j S_jS_j^*.\]
  The $n\times n$ diagonal matrices~$X=\diag(\sqrt a_j)$
  and~$Y=\diag(\sqrt b_j)$ have
  \[ RX(RX)^* = SY(SY)^*,\] so there is a unitary matrix~$U$
  with~$RX=SYU$. (Indeed,~$B$ and~$Y$ are both invertible, so~$SY$ is
  invertible and~$U=RX(SY)^{-1}$ has real entries and is an orthogonal
  matrix).  As shown in~\cite{ang-cow-nar}, this implies that the
  factorisation~$B=S^*R$ attains the Haagerup bound. Indeed, the unit
  vectors~$x=[\sqrt a_j]_{1\leq j\leq n}$ and~$y=[\sqrt{b_j}]_{1\leq
    j\leq n}$ satisfy
  \begin{align*}
    \langle (B\schur U^t)x,y\rangle &= \trace((SY)^*RXU) =
  \trace(S^*SYY^*) \\&= \sum_{j=1}^n \|S_j\|^2|y_j|^2 =
  c(S)^2=c(S)c(R),
\end{align*}
so by Proposition~\ref{prop:upperbound},
\begin{align*}
  \tfrac2{n+1}\cot\theta_{n+1} = c(S)c(R) &\leq \|B\schur U^t\|\\&\leq
\|B\|_{\schur} \leq \|\tilde B\|_\schur \leq c(\tilde S)c(\tilde
R)=\tfrac 2{n+1}\cot\theta_{n+1}.
\end{align*}
Hence $\|\Sigma(n,n)\|=\|B\|_\schur=\|\Sigma(n,n+1)\|=\|\tilde
B\|_\schur=\frac 2{n+1}\cot \theta_{n+1}$.
\end{proof}
\goodbreak

\section{Calculations and estimates of small norms}\label{sec:calcs}

In this section, we calculate or estimate the norms of some particular
idempotent Schur multipliers. Our first result is
Proposition~\ref{prop:smallnorms}, in which we find the exact norms of
some idempotent Schur multipliers in low dimensions.  We then find
lower bounds for the norms of some other Schur idempotents
which we will use to establish Theorem~\ref{thm:gaps} in the following
section.%

\begin{prop}\label{prop:smallnorms}\leavevmode
  \begin{enumerate}
  \item $\|\gqqq\|=\eta_1=1$.
  \item $\|\snakeww\|=\|\snakewe\|=\eta_2=\sqrt{4/3}\approx 1.15470$.

  \item $\|\geeq\|=\|\snakeer\|=\|\looprr\|=\eta_3=\frac{1+\sqrt2}2\approx 1.20711$.

  \item $\|\geew\|=\|\getq\|=\eta_4=\frac1{15}\sqrt{169+38\sqrt{19}}\approx 1.21954$.

  \item\label{prop:smallnorms:sqrt32}
    $\|\grew\|=\eta_5=\sqrt{3/2}\approx 1.22474$.

  \item $\|\snakerr\|=\|\snakert\|=\eta_6=\tfrac25\sqrt{5+2\sqrt5}\approx 1.23107$.

  \item\label{prop:smallnorms:trie}
    $\|\geee\|=\frac1{15}(9+4\sqrt6)\approx 1.25320$.

  \item $\|\geer\|=9/7\approx 1.28571$.

  \item $\|\geet\|=4/3\approx 1.33333$.
  \end{enumerate}
\end{prop}
\begin{proof}
  (1) is trivial, and assertions (2), (3) and~(6) are consequences of
  Theorem~\ref{thm:snakes} and equation~\eqref{eq:davdon-loops}.
  \begin{enumerate}\addtocounter{enumi}{3}
    
  \item 

    Let $B=\begin{smallbmatrix}
      1 & 0 & 0 & 1 & 0 \\
      1  & 1 & 1 &  0 & 0 \\
      0 & 0 & 1 & 0 & 1
    \end{smallbmatrix}\simeq M(\getq)$.
    Consider the matrices \[ %
    P=\begin{bmatrix}
      \eta  & \alpha  & \beta\\
      \alpha  & \eta  & \alpha\\
      \beta  & \alpha  & \eta \\
    \end{bmatrix}\quad\text{and}\quad
    Q=
    \begin{bmatrix}
      \eta  & \gamma  & \delta  & \sigma  & \tau  \\
      \gamma  & \eta  & \gamma  & -\sigma  & -\sigma  \\
      \delta  & \gamma  & \eta  & \tau  & \sigma  \\
      \sigma  & -\sigma  & \tau  & 2 \sigma  & \alpha  \\
      \tau & -\sigma & \sigma & \alpha & 2 \sigma
    \end{bmatrix}
    \] where $\eta=\eta_4$ and
  \begin{align*}
  \alpha&=  \frac{1}{15} \sqrt{139-22 \sqrt{19}},&
  \qquad
  \beta&=-\frac{1}{15} \sqrt{24-2 \sqrt{19}},\\
  \gamma&=\frac{2}{15} \sqrt{16+2\sqrt{19}},&
  \delta&=\frac{1}{15} \sqrt{424-82 \sqrt{19}},\\
  \sigma&=\frac1{15}\sqrt{61+2 \sqrt{19}},
  &
  \tau&=-\frac{1}{15} \sqrt{256-58 \sqrt{19}}.
\end{align*}
One can check with a computer algebra system that~$C=\left[\begin{smallmatrix} P & B \\ B^* & Q
  \end{smallmatrix}\right]$ has rank~$3$ and
its non-zero eigenvalues are positive, so~$C$ is positive
semidefinite.  The maximum diagonal entry of~$C$
is $\max\{\eta,2\sigma\}=\eta$, so~$\|\getq\|\leq \eta$ by~\cite{pps}
(see also~\cite[Exercise~8.8(v)]{paulsen-book}).

On the other hand,
    \[ U=\frac1{15}
    \begin{bmatrix}
        8+\sqrt{19} & -\sqrt{74-2 \sqrt{19}} & -7+\sqrt{19} \\
        \sqrt{74-2 \sqrt{19}} & 1+2 \sqrt{19} & \sqrt{74-2 \sqrt{19}} \\
        -7+\sqrt{19} & -\sqrt{74-2 \sqrt{19}} & 8+\sqrt{19}
      \end{bmatrix}
    \]
    is orthogonal, and if~$B=
    \begin{smallbmatrix}
        1 & 0 & 0\\
  1  & 1 & 1 \\
  0 & 0 & 1
    \end{smallbmatrix}\simeq M(\geew)$, then
    $\|B\schur U\|=\eta\leq \|\geew\|$. Since~$\geew\leq \getq$, this
    shows that $ \eta\leq \|\geew\|\leq \|\getq\|\leq \eta$ and we
    have equality.

  \item 
    Let
    \[
    S=\frac1{2\cdot 54^{1/4}}
    \begin{smallbmatrix}
      2\sqrt6&2\sqrt6&2\sqrt6\\-2\sqrt3&\sqrt3&\sqrt3\\
      0&3&-3
    \end{smallbmatrix}
    \quad\text{and}\quad
    R=\frac1{54^{1/4}}
    \begin{smallbmatrix}
      3&1&1&1\\
      0&-2\sqrt2&\sqrt2&\sqrt2\\
      0&0&\sqrt6&-\sqrt6    
    \end{smallbmatrix}.
    \]
    Then $S^*R=\begin{smallbmatrix}
        1&1&0&0\\1&0&1&0\\1&0&0&1
      \end{smallbmatrix}\simeq M(\grew)$, so $\|\grew\|\leq
    c(S)c(R)=\sqrt{3/2}$. On the other
    hand, consider
    \[ V=
    \frac14\begin{bmatrix}
      \sqrt 5 & 3 & -1 & -1 \\
      \sqrt 5 & -1 & 3 & -1 \\
      \sqrt 5 & -1 & -1 & 3 
    \end{bmatrix}.
    \]
    It is easy to see that~$V$ is a coisometry with $
    \|\begin{smallbmatrix}
        1&1&0&0\\1&0&1&0\\1&0&0&1
      \end{smallbmatrix}\schur V\|=\sqrt{3/2}$. Hence
    $\|\grew\|\geq \sqrt{3/2}$.

    \stepcounter{enumi}
  \item Consider
    \[ S=
    \begin{bmatrix} 1&1&1/2\\a&-a&b\\-a&a&c
    \end{bmatrix} \quad\text{and}\quad R=
    \begin{bmatrix} 1&1&1/2\\a&-a&b\\a&-a&-c
    \end{bmatrix}
    \] where
    \[a=\sqrt{\tfrac1{15}(-3+2\sqrt6)},\ b=\tfrac{1}{2}
    \sqrt{\tfrac{1}{15} (3+8 \sqrt{6})}\text{ and }
    c=\sqrt{\tfrac{1}{30} (9+4 \sqrt{6})}.\] Since $a(b+c)=\frac12$ and $b^2-c^2=-\tfrac14$, we have $S^*R=
    \begin{smallbmatrix} 1&1&1\\1&1&0\\1&0&0
      \end{smallbmatrix}\simeq M(\trie)$, so $\|\trie\|\leq
    c(S)c(R)=\frac1{15}(9+4\sqrt6)$.  On the other hand, calculations may
    be performed to show that the matrix
    \[ U=\frac1{15} \begin{bmatrix}
       9-\sqrt{6} & \sqrt{54-6 \sqrt{6}} & 2 \sqrt{21
          + 6 \sqrt6} \\ \sqrt{54-6 \sqrt{6}} & 3 (1+\sqrt{6}) & -2
        \sqrt{27-3 \sqrt{6}} \\ 2 \sqrt{21 + 6 \sqrt6} & -2 \sqrt{27-3
          \sqrt{6}} & 3-2 \sqrt{6}
      \end{bmatrix}
      \] is orthogonal, and $\|
    \begin{smallbmatrix} 1&1&1\\1&1&0\\1&0&0
      \end{smallbmatrix}\schur U\|=\frac1{15}(9+4\sqrt6)$.

  \item
    Let
    \[ S=\frac1{\sqrt{14}}
      \begin{bmatrix}
        3 & 4 & 3 \\
        -\sqrt{2} & \sqrt{2} & -\sqrt{2} \\
        -\sqrt{7} & 0 & \sqrt{7}
      \end{bmatrix}
    \quad\text{and}\quad
    R=\frac1{\sqrt{14}}
      \begin{bmatrix}
        3 & 4 & 3 \\
        \sqrt{2} & -\sqrt{2} & \sqrt{2} \\
        -\sqrt{7} & 0 & \sqrt{7}
      \end{bmatrix}
    .\]
    Then $S^*R=
    \begin{smallbmatrix}
      1&1&0\\1&1&1\\0&1&1
    \end{smallbmatrix}=M(\geer)$, so $\|\geer\|\leq c(S)c(R)=9/7$.
  The matrix
    \[ U=\frac17
      \begin{bmatrix}
        3 & 2 \sqrt{6} & -4 \\
        2 \sqrt{6} & 1 & 2 \sqrt{6} \\
        -4 & 2 \sqrt{6} & 3
      \end{bmatrix}
    \]
    is orthogonal, and $\|\begin{smallbmatrix}
      1&1&0\\1&1&1\\0&1&1
    \end{smallbmatrix}\schur U\|=9/7\leq \|\geer\|$.
  \item This is proven in~\cite[Theorem~2.1]{bcd}, and is also a
    consequence of equation~\eqref{eq:davdon-loops}.\qedhere
  \end{enumerate}
\end{proof}

\begin{rk}
  Since $M(\geee)\simeq
\begin{smallbmatrix}
        1&1&1\\0&1&1\\0&0&1
      \end{smallbmatrix}$, part~(\ref{prop:smallnorms:trie}) of the preceding result
    gives the norm of the upper-triangular truncation map on the
    $3\times 3$ matrices. This result has previously been stated in
    \cite[p.~131]{ang-cow-nar}, but a detailed calculation does not
    appear in that reference.
\end{rk}

\begin{rk}\label{rk:sqrt2}
  Proposition~\ref{prop:smallnorms}(\ref{prop:smallnorms:sqrt32}) may be generalised to show that
  \[\|[\mathbf 1\ I_n]\|_{\schur}=\sqrt{\frac{2n}{n+1}}\]
  where~$\mathbf 1$ is the $n\times 1$ vector of all ones and~$I_n$ is
  the $n\times n$ identity matrix.
  We omit the details.
\end{rk}

\begin{prop}\label{prop:grrc}
  $\|\grrc\|>\|\trie\|$.
\end{prop}
\begin{proof}
  The matrix~$U=\displaystyle\frac16
  \begin{smallbmatrix}
    3&3&3&3\\
    3&-5&1&1\\
    3&1&-5&1\\
    3&1&1&-5
  \end{smallbmatrix}$ has~$U=2P-I$ where~$P$ is the rank one
  projection onto the linear span of~$\begin{smallbmatrix}
    3\\1\\1\\1
  \end{smallbmatrix}$,  so~$U$ is orthogonal. Clearly~$B=
  \begin{smallbmatrix}
    1&1&1&1\\
    0&1&0&0\\
    0&0&1&0\\
    0&0&0&1
  \end{smallbmatrix}$ has $M(\grrc)\simeq B$, and~$Y=6B\schur U$ has
  \begin{align*}
    \|Y\|^2&=%
    \left\|\begin{smallbmatrix}
        3&3&3&3\\
        0&-5&0&0\\
        0&0&-5&0\\
        0&0&0&-5
      \end{smallbmatrix}
      \begin{smallbmatrix}
        3&0&0&0\\
        3&-5&0&0\\
        3&0&-5&0\\
        3&0&0&-5
      \end{smallbmatrix}\right\|=\left\|
      25I + 
      \begin{smallbmatrix}
        11&-15&-15&-15\\-15&0&0&0\\-15&0&0&0\\-15&0&0&0
      \end{smallbmatrix}\right\|.
  \end{align*}
  Since the norm of~$YY^*$ is its spectral radius, a 
  calculation gives
  \[ \|Y\|^2=\tfrac12(61+\sqrt{2821}).\]
  Hence
    \[\|\grrc\|=\|B\|_{\schur}\geq \tfrac16\|Y\| = \tfrac16\sqrt{\tfrac12(61+\sqrt{2821})}>\tfrac1{15}(9+4\sqrt6) = \|\trie\|.\qedhere\]
\end{proof}

\begin{prop}
  \label{prop:grrx}
  $\|\grrx\|>\|\snakerr\|$.
\end{prop}
\begin{proof}
  Consider the unit vectors~$x$ and~$y$ appearing in the proof of
  Theorem~\ref{thm:snakes} in the case~$n=4$. It turns out that
  \[ x=\frac1{\sqrt5}
  \begin{bmatrix}
    \sqrt{\tfrac12(3-\sqrt5)}\\
    \sqrt{\frac12(1+\sqrt5)}\\
    \sqrt2\\
    1
  \end{bmatrix}\quad\text{and}\quad
  y=\frac1{\sqrt5}
  \begin{bmatrix}
    1\\\sqrt{2}\\
    \sqrt{\frac12(1+\sqrt5)}\\
    \sqrt{\tfrac12(3-\sqrt5)}
  \end{bmatrix},\]
  and that the matrix~$B=
  \begin{smallbmatrix}
    1&1&0&0\\
    0&1&1&0\\
    0&0&1&1\\
    0&0&1&0
  \end{smallbmatrix}\simeq M(\grrx)$ has 
  $\|B^t\schur (xy^*)\|_1>1.235>\|\snakerr\|$, where~$\|\cdot\|_1$ is
  the trace-class norm. It is well-known and easy to see
  that~$S_B\colon \B\to \B$, $A\mapsto B\schur A$ is the dual
  of~$T_{B^t}\colon \C_1\to \C_1$, $C\mapsto B^t\schur C$, the mapping
  of Schur multiplication by~$B^t$ on the trace-class operators~$\C_1$
  (viewed as the predual of~$\B$). Since~$x$ and~$y$ are unit vectors,
  $\|xy^*\|_1=1$ and so
  $\|B\|_{\schur}=\|S_B\|=\|T_{B^t}\|>\|\snakerr\|$.
\end{proof}

\begin{prop}\label{prop:grrr}
  $\|\grrr\|>\|\snakerr\|$.
\end{prop}
\begin{proof}
  Consider
  \[ B=
  \begin{bmatrix}
    1&1&0&0\\1&0&1&0\\1&0&0&1\\0&0&0&1
  \end{bmatrix}\quad\text{and}\quad
  U=\frac14\begin{bmatrix}
    \sqrt5&3&-1&-1\\\sqrt5&-1&3&-1\\\sqrt5&-1&-1&3\\-1&\sqrt5&\sqrt5&\sqrt5
  \end{bmatrix}.\] The matrix~$U$ is orthogonal, and $M(\grrr)\simeq
  B$.  Now
  \begin{align*}
    16\|B\schur U\|^2&=\left\|
      \begin{bmatrix}
        15&3\sqrt5&3\sqrt5&3\sqrt5\\3\sqrt5&9&0&0\\3\sqrt5&0&9&0\\3\sqrt5&0&0&14
      \end{bmatrix}\right\| = \left\|9I + Z%
    \right\|
  \end{align*}
  where~$Z=\begin{smallbmatrix}
    6&\sqrt5&\sqrt5&\sqrt5\\\sqrt5&0&0&0\\\sqrt5&0&0&0\\\sqrt5&0&0&5
  \end{smallbmatrix}$, which has characteristic
  polynomial $p(x)=x(x^3-11x-105x+450)$. By estimating the roots
  of~$p(x)$, one can show that~$\|B\schur U\|=\frac14\sqrt{9+\lambda}$
  where~$\lambda$ is the largest root of~$p(x)$, and that
  $\|\grrr\|\geq \|B\schur U\|>\|\snakerr\|$.
\end{proof}

\begin{prop}\label{prop:grrz}
  $\|\grrz\|> \|\snakerr\|$.
\end{prop}
\begin{proof}
  Consider the symmetric matrices
  \[ B= \begin{bmatrix}
    0&0&0&1\\0&1&1&1\\0&1&0&0\\1&1&0&0
  \end{bmatrix} \quad\text{and}\quad U=
  \begin{bmatrix}
    0&0&-1/\sqrt2&1/\sqrt2\\
    0&1/3&2/3&2/3\\
    -1/\sqrt2 &2/3&-1/6&-1/6\\
    1/\sqrt2 &2/3&-1/6&-1/6
  \end{bmatrix}.
  \] By direct calculation,~$U$ is orthogonal, and 
  $M(\grrz)\simeq B$. The characteristic polynomial of~$B\schur U$
  is~$p(x)=\frac1{18}(x+1)(18 x^3-24x^2-x+4)$. It is easy to see
  that~$p(x)$ has two negative roots and two positive roots, and the
  smallest root is~$-1$ while the largest root is larger
  than~$1$. Since~$B\schur U$ is symmetric,~$\|B\schur U\|$ is the
  spectral radius of~$p(x)$, which is the largest root
  of~$p(x)$. But~$p(\|\snakerr\|)<0$ and~$p'(x)>0$ for~$x>1$, so
  $\|\grrz\|\geq \|B\schur U\|>\|\snakerr\|$.
\end{proof}

\begin{rk}
  Numerical methods produce the following estimates for these norms,
  each correct to $5$ decimal places: $\|\grrx\|\approx 1.24131$,
  $\|\grrr\|\approx 1.25048$, $\|\grrz\|\approx 1.25655$ and
  $\|\grrc\| \approx 1.25906$.  To see this, we apply the numerical
  algorithm described in~\cite{cowen-et-al} to~$M(G)$ for each of
  these graphs~$G$. The algorithm requires a unitary matrix without
  zero entries as a seed.  Using the $4\times 4$ Hadamard unitary
  $H_4=H_2\otimes H_2$ where
  $H_2=\frac1{\sqrt2}\begin{smallbmatrix}1&1\\1&-1
    \end{smallbmatrix}$, after 20 or fewer
  iterations, in each case the algorithm produces real matrices~$R$
  and~$S$ for which the Haagerup estimate gives an upper
  bound~$\beta=c(S)c(R)$, and an orthogonal matrix~$U$ giving a lower
  bound~$\alpha=\|M(G)\schur U\|$, so that
  $\beta-\alpha<10^{-6}$. 
\end{rk}

\section{A characterisation of the Schur idempotents with small norm}\label{sec:proofmain}

We now use the results of the previous section to characterise the
Schur idempotents with norm~$\eta_k$ for $1\leq k\leq 6$. This will
yield a proof of Theorem~\ref{thm:gaps}. %

\begin{notn}
  We will write \[\Gamma=\bigcup_{1\leq m,n\leq \aleph_0}\Gamma(m,n).\] 
  Note that
  $\N=\{\|G\|\colon G\in \Gamma\}\setminus\{\infty\}$.
\end{notn}

\begin{rk}
  In the arguments below, we frequently encounter the following
  situation: $G$ is a twin-free bipartite graph with an induced
  subgraph~$H$, and~$H$ contains two vertices $v_1$ and~$v_2$ which
  are twins (in~$H$).  Since~$G$ is twin-free, we can conclude that
  there is a vertex~$w$ in~$G$ which is joined to one of~$v_1$
  and~$v_2$ but not the other. We will say that the vertex~$w$
  \emph{distinguishes} the vertices~$v_1$ and $v_2$.
\end{rk}

\def\gwem{\gwe{\nn11\nn12\nn21\nn22\nn23}}
\def\gqrq{\gqr{\nn21\nn22\nn23\nn24}}
\begin{lem}\label{lem:deg3b}
  Let~$G\in\Gamma$ be twin-free.
  \begin{enumerate}
  \item If $\deg(G)\geq 3$, then $G$ contains either $\geew$, $\geee$
    or~$\geer$ as an induced subgraph.
  \item If~$1<\|G\|<\eta_4$, then~$\deg(G)=2$.
  \end{enumerate}
\end{lem}
\begin{proof}
  (1) Let~$v$ be a vertex in $G$ of degree at least~$3$ and consider
  an induced subgraph $\kqe$ with~$v$ at the top. Since~$G$ is
  twin-free, it is not hard to see that there are at least two other
  row vertices in~$G$ which distinguish the neighbours of~$v$, and
  that this necessarily yields one of the induced subgraphs in the statement.

  (2) follows from~(1), since $\geew$, $\geee$ and~$\geer$ all have
  norm at least~$\eta_4$.
\end{proof}

\begin{lem}\label{lem:deg2}
  If~$G\in \Gamma$ is connected with~$\deg(G)=2$ and $\|G\|<4/\pi$, then
  $\|G\|=\|\Sigma(n,n)\|$ for some unique $n\ge2$. Moreover, \[E\leq
  G\leq F\] where $E=\Sigma(n,n)$ and
  \[F=
  \begin{cases}
    \Sigma(n,n+1)&\text{if~$n$ is even,}    \\
    \Lambda(n+1)&\text{if~$n$ is odd.}
  \end{cases}\]
\end{lem}
\begin{proof}
  The graph~$G$ is connected and $\deg(G)=2$, so~$G$ is either a path
  or a cycle. Since the sequence~$\frac2{n}\cot\theta_n$ is strictly
  increasing with limit~$4/\pi$ and~$\frac2n\csc\theta_n>4/\pi$ for
  every~$n$, the claim follows from equation~\eqref{eq:davdon-loops}
  and Theorem~\ref{thm:snakes}.
\end{proof}

\begin{lem}\label{lem:deg3}
  If~$G\in\Gamma$ is twin-free with $\|G\|< \|\trie\|$, then
  \begin{enumerate}
  \item  $G\not \geq\gwem$ and
  \item $\deg(G)\leq 3$.
  \end{enumerate}
\end{lem}
\begin{proof}
  (1) Otherwise, since~$G$ is twin-free, there is a row
  vertex~$r$ in~$G$ which distinguishes the twin column
  vertices %
  in $\gwem$. %
  Hence either
  $G\geq\trie$ or $G\geq \geer$, and so $\|G\|\geq \|\trie\|$ by
  Proposition~\ref{prop:smallnorms}.
  
  (2) Suppose that $\deg(G)>3$, so that $G\geq \gqrq$. In order to
  distinguish between the four twin column vertices, there must
  be another row vertex in~$G$ attached to one but not all of these,
  so $G\supseteq\gwr{\nn11\nn21\nn22\nn23\nn24}$. In fact, to avoid
  the induced subgraph forbidden by~(1), we must have
  $G\geq\gwr{\nn11\nn21\nn22\nn23\nn24}$. Distinguishing between the
  remaining columns using the same argument shows that $G\geq\grrc$,
  so $\|G\|\geq \|\grrc\|>\|\trie\|$ by Proposition~\ref{prop:grrc},
  contrary to hypothesis.
\end{proof}

\begin{notn}
  We define graphs~$E_j\leq F_j$ for~$1\leq j\leq 6$ by:
  \begin{equation*}
    E_1=F_1=\gqqq,\qquad E_2=\gwwq,\ F_2=\gweq,\qquad E_3=\snakeee,\ F_3=\looprr
  \end{equation*}
  \begin{equation*}
    E_4=\geew,\ \!F_4=\getq,\quad\! E_5=F_5=\grew,\quad\! E_6=\snakerr,\ \!F_6=\snakert
  \end{equation*}
  Note that $\|E_j\|=\|F_j\|=\eta_j$ for $1\leq j\leq 6$.
\end{notn}

\begin{thm}\label{thm:characterisations}
  Let $G\in \Gamma$ be a twin-free, connected bipartite graph. For
  each~$k\in \{1,2,3,4,5,6\}$, the following are equivalent:
  \begin{enumerate}
  \item $E_k\leq G\leq F_k$;
  \item $\|G\|=\eta_k$;
  \item $\eta_{k-1}<\|G\|\leq \eta_{k}$.
  \end{enumerate}
\end{thm}
\begin{proof}
  For each~$k$, the implication $(1)\implies(2)$ follows from
  Propositions~\ref{prop:basic-graphs} and~\ref{prop:smallnorms}, and
  $(2)\implies(3)$ is trivial. 

  Suppose that~$G$ satisfies~(3).

  If~$k=1$, then $0<\|G\|\leq 1$, so $\|G\|=1$ and~$G$
  is a disjoint union of complete bipartite graphs
  by~\cite[Theorem 4]{kat-paulsen}. Since~$G$ is connected and
  twin-free, $G=\gqqq$.

  If~$k\in \{2,3\}$, then $\deg(G)=2$ by Lemma~\ref{lem:deg3b}, so
  $E_k\leq G\leq F_k$ by Lemma~\ref{lem:deg2}.

  If~$k\in \{4,5,6\}$ but $E_6\ne G\ne F_6$, then $\deg(G)\ne2$ by
  Lemma~\ref{lem:deg2} and $\deg(G)\leq 3$ by Lemma~\ref{lem:deg3}, so
  $\deg(G)=3$. Since $\|G\|<\|\trie\|<\|\geer\|$, we have
  $E_4=\geew\leq G$ by Lemma~\ref{lem:deg3b}.

  If~$G$ has the same row vertices as~$E_4$, then any column
  vertex~$c$ in~$G$ which is not in~$E_4$ must be joined to~$E_4$ so
  as to avoid the induced subgraph~$\loopee$, and $c$ cannot be joined
  to the degree~$3$ vertex of~$E_4$ since~$\deg(G)=3$. Hence~$c$ must
  be joined to precisely one of the row~vertices of~$E_4$ with degree
  one. Since~$G$ is twin-free, this gives~$G\leq \getq=F_4$, so
  $E_4\leq G\leq F_4$.

  If on the other hand~$G$ has at least four row vertices, choose a
  row vertex of~$G$ of smallest possible distance~$\delta\in \{1,2\}$
  to the induced subgraph $E_4\leq G$. 
  If $\delta=2$, then $\grrx\subseteq G$, and the rightmost row vertex $r_4$
  of~$\grrx$ is not connected to any of $c_1,c_2,c_3$ in~$G$. Since
  $\grrx$ is not an induced subgraph of~$G$ by
  Proposition~\ref{prop:grrx} and $\deg(G)=3$, we have $G\geq
  \grr{\vrrx\nn14}$; but removing the two degree~$1$ vertices then
  shows that~$G$ contains the forbidden induced subgraph~$\loopee$, a
  contradiction.

  So~$\delta=1$. We claim that~$G=E_5$. Indeed, since~$\delta=1$ we
  know that one of the following is an induced subgraph of~$G$: \[
  G_1=\gre{\veew\nn43} 
  \quad\! G_2=E_5=\grew %
  \quad G_3=\gre{\nn11\nn21\nn22\nn23\nn33\nn43\nn42} 
  \quad\! G_4=\gre{\veew\nn43\nn41} 
  \quad G_5=\gre{\veew\nn43\nn42\nn41}
  \]
  Observe that $\trie$ is an induced subgraph of both $G_3$ and $G_4$,
  so the norms of these are too large. We can also rule out~$G_5$
  since it has a pair of twin row vertices of degree~$3$, so these
  cannot be distinguished in~$G$. If~$G_1\leq G$, then since the
  vertices~$r_3$ and $r_4$ are twins in~$G_1$ but not in~$G$, there is
  a column vertex $c_4$ attached to~$r_4$ (say) but not $r_3$. We
  cannot join $c_4$ to the maximal degree vertex~$r_2$, so we find
  that either
  \[ 
  \grr{\veew\nn43\nn44}
  \quad \text{or}\quad
  \grr{\veew\nn43\nn44\nn14}
  \]
  is an induced subgraph of~$G$ containing~$G_1$. However, the first
  is ruled out by Proposition~\ref{prop:grrz} and the second contains
  an induced subgraph $\loopee$, so cannot occur either. So~$E_5\leq
  G$. If~$E_5$ is a proper induced subgraph of~$G$, then since we must
  avoid~$\loopee$ and also the induced subgraph~$\grrr$ by
  Proposition~\ref{prop:grrr}, it follows that no column vertex of~$G$
  has distance~$1$ to~$E_5$. So there is a row vertex of~$G$ with
  distance~$1$ to~$E_5$. Avoiding $\trie$ and twin vertices of
  degree~$3$, we find an induced subgraph
  $\gte{\nn12\nn22\nn32\nn33\nn43\nn34\nn54}\leq G$. To distinguish
  between the first two row vertices, we add a column vertex while
  avoiding~$\loopee$, and conclude
  that~$\gtr{\nn11\nn12\nn22\nn32\nn33\nn43\nn34\nn54}\leq
  G$. Removing one row vertex gives $\grrz\leq G$, contradicting
  Proposition~\ref{prop:grrz}.

  In summary: if~$k=4$, then $E_4\leq G\leq F_4$; if~$k=5$, then
  $G=E_5$; and if~$k=6$ then~$E_6\leq G\leq F_6$.
\end{proof}

Theorem~\ref{thm:gaps} is an immediate consequence of
Theorem~\ref{thm:characterisations} and
Proposition~\ref{prop:basic-graphs}. We also obtain: %

\begin{cor}\label{cor:characterisations}
  Let~$k\in \{1,2,3,4,5,6\}$.
  \begin{enumerate}
  \item If~$G\in \Gamma$ is twin-free and connected, then
    \[ \|G\|\leq \eta_k\iff G\leq F_j\text{ for some~$j\leq k$}.\]
  \item If~$G\in\Gamma$, then $\|G\|=\eta_k$ if and only if:
    \begin{enumerate}
    \item each component~$H$ of~$G$ satisfies~$\df(H)\leq F_j$ for
    some~$j\leq k$; and
  \item there is a component~$H$ of~$G$ with~$E_k\leq \df(H)$.
    \end{enumerate}
  \end{enumerate}
\end{cor}

\section{Normal masa bimodule projections}\label{sec:cbn}

\newcommand{\cbn}{\mathit{NCB}_\D(\B(\H))}

Let~$\H$ be a separable Hilbert space. Given a masa (maximal
abelian selfadjoint subalgebra)~$\D\subseteq \B(\H)$, we write~$\cbn$
for the set of normal completely bounded linear maps~$\B(\H)\to\B(\H)$
which are bimodular over~$\D$.  Smith's theorem~\cite{smith91} ensures
that $\|\Phi\|=\|\Phi\|_{cb}$ for any~$\Phi\in \cbn$. Moreover,
by~\cite[Theorem~2.3.7]{SS-masas}, there is a standard finite measure
space~$(X,\mu)$ so that $\D$ is unitarily equivalent
to~$L^\infty(X,\mu)$ acting by multiplication on~$L^2(X,\mu)$. Hence
we will take $\D=L^\infty(X,\mu)$ and~$\H=L^2(X,\mu)$ without loss of
generality.

Recall that a set~$R\subseteq X\times X$ is marginally null
if~$R\subseteq (N\times X)\cup (X\times N)$ for some null
set~$N\subseteq X$. Two Borel functions~$\phi,\psi\colon X\times X\to
\bC$ are equal marginally almost everywhere (m.a.e.) if
$\{(x,y)\in X\times X\colon \phi(x,y)\ne \psi(x,y)\}$ is marginally
null. We write $[\phi]$ for the equivalence class of all Borel
functions which are equal m.a.e.\ to~$\phi$. Let $L^\infty(X,\ell^2)$
denote the Banach space of essentially bounded measurable functions $X\to
\ell^2$, identified modulo equality almost everywhere. For $f,g\in
L^\infty(X,\ell^2)$, we write $\langle f,g\rangle\colon X\times X\to \bC$ for
the function given m.a.e.\ by $\langle f,g\rangle(s,t)=\langle
f(s),g(t)\rangle$. As shown in~\cite{kat-paulsen}, there is a
bijection
\[ \Gamma\colon \cbn \to \{ [\langle f,g\rangle]\colon f,g\in
L^\infty(X,\ell^2)\}\] so that for every~$\phi\in \Gamma(\Phi)$, the
map~$\Phi$ is the normal extension to~$\B(\H)$ of pointwise
multiplication by~$\phi$ acting on the (integral kernels of)
Hilbert-Schmidt operators in~$\B(\H)$. Moreover, $\Gamma$ is a
homomorphism with respect to composition of maps and pointwise
multiplication, and
\[ \|\Phi\|=\inf\{ \|f\|\,\|g\|\colon f,g\in L^\infty(X,\ell^2),\
\Gamma(\Phi)=[\langle f,g\rangle]\}\] and this infimum is attained. In
the discrete case, this reduces to~\cite[Corollary~8.8]{paulsen-book}.

\begin{lem}\label{lem:direct-sums-cts}
  Let~$\Phi\in \cbn$. If~$\Gamma(\Phi)=[\phi]$ and~$\{R_j\}_{j\ge1}$,
  $\{C_j\}_{j\ge1}$ are two countable Borel partitions of~$X$ with
  $\phi^{-1}(\bC\setminus\{0\})\subseteq \bigcup_{j\ge1} R_j\times
  C_j$, then $\|\Phi\|=\sup_j\|\Phi_j\|$
  where~$\Gamma(\Phi_j)=[\chi_{R_j\times C_j}\cdot\phi]$.
\end{lem}
\begin{proof}
  Let~$P_j=\chi_{R_j}$ and~$Q_j=\chi_{C_j}$. Note that~$\{P_j\}$
  and~$\{Q_j\}$ are then partitions of the
  identity in~$\D$. By~\cite[Theorem~10]{kat-paulsen}, the map~$\Psi$ given
  by~$\Psi(T)=\sum_{j\ge1} P_j TQ_j$ is in~$\cbn$, and
  \[\Gamma(\Psi)=[\chi_{K}]\quad\text{where}\quad K=\bigcup_{j\ge1}R_n\times
  C_n.\] Since~$\Gamma$ is a homomorphism and $\phi=\chi_K\cdot\phi$, we have
  \[ \Gamma(\Phi)=\Gamma(\Psi)\cdot \Gamma(\Phi)=\Gamma(\Psi\circ \Phi),\] hence
  $\Phi=\Psi\circ \Phi$. Let~$\Psi_j\in\cbn$ be given
  by~$\Psi_j(T)=P_j TQ_j$, and let~$\Phi_j=\Psi_j\circ
  \Phi$. Since~$\Gamma$ is a homomorphism,
  $\Gamma(\Phi_j)=\Gamma(\Psi_j\circ \Phi)=[\chi_{R_j\times C_j}\cdot
  \phi]$, and for any~$T\in \B(\H)$,
  \[ \|\Phi(T)\|=\|\Psi\circ \Phi(T)\|=\sup_{j\ge1} \|P_j\Phi(T)Q_j\|=\sup_{j\ge1} \|\Phi_j(T)\|.
  \qedhere
  \]
\end{proof}

\begin{prop}\label{prop:biggraph}
  Let~$\Phi\in\cbn$ be idempotent and let $\eta>\|\Phi\|$.
  \begin{enumerate}
  \item There exist a Borel set~$G\subseteq X\times X$ and weakly
    Borel measurable functions $f,g\colon X\to \ell^2$ so that
    \begin{enumerate}
    \item $\Gamma(\Phi)=[\chi_G]$;
    \item $\chi_G(x,y)=\langle f(x),g(y)\rangle$ for all $x,y\in
      X$; and
    \item $\sup_{x,y\in X} \|f(x)\|\,\|g(y)\|<\eta$.
    \end{enumerate}
  \item For such a set~$G$, there are two countable families of disjoint
    Borel subsets of~$X$, say~$\{R_j\}$ and~$\{C_j\}$, so that the
    components of~$G$ are the Borel sets $G_j=G[R_j,C_j]$, and there are
    maps~$\Phi_j\in \cbn$ with $\Gamma(\Phi_j)=[\chi_{G_j}]$ and
    $\|\Phi\|=\sup_j \|\Phi_j\|$.
  \item If~$F$ is a countable induced subgraph of~$G$,
    then~$\|F\|<\eta$.
  \item If~$\df(G)$ is countable, then~$\|\Phi\|\leq \|\df(G)\|$.
  \end{enumerate}
\end{prop}
\begin{proof}
  \begin{enumerate}
  \item We have~$\Phi=\Phi\circ \Phi$, and~$\Gamma$ is a
    homomorphism. Hence if~$\phi\colon X\times X\to \bC$ is Borel
    with~$\Gamma(\Phi)=[\phi]$, then
    $[\phi]=\Gamma(\Phi)=\Gamma(\Phi)^2=[\phi^2]$, from which it
    follows that $[\phi]=[\chi_{G}]$ where~$G$ is the Borel
    set~$G=\phi^{-1}(1)$. Hence there are~$f,g\in L^\infty(X,\ell^2)$
    with $[\chi_G]=[\langle f,g\rangle]$
    and~$\|f\|\,\|g\|=\|\Phi\|<\eta$. Multiplying~$f$ and~$g$ by
    $\chi_{X\setminus N}$ for some null set~$N$ and removing the
    marginally null set $(N\times X)\cup (X\times N)$ from~$G$, we can
    achieve both pointwise equality $\chi_G=\langle f,g\rangle$
    on~$X\times X$ and $\sup_{x,y\in X}\|f(x)\|\,\|g(y)\|<\eta$.

    \item As in~\cite{kat-paulsen}, we can use the following argument of
    Arveson to show that~$G$ is a countable union of Borel
    rectangles. Since~$\ell^2$ is separable, the open
    set~$\{(\xi,\eta)\in\ell^2\times \ell^2\colon \langle
    \xi,\eta\rangle\ne0 \}$ is a countable union $\bigcup_{n\ge1}
    U_n\times V_n$ where~$U_n,V_n$ are open subsets
    of~$\ell^2$. Let~$A_n=f^{-1}(U_n)$ and~$B_n=g^{-1}(V_n)$. These
    are Borel sets, and $G=\bigcup_{n\ge1} A_n\times B_n$. Discard
    empty sets, so that~$A_n,B_n\ne\emptyset$ for all $n\ge1$.

    For each~$j\in\bN$, the component of~$G$ containing~$A_j$
    and~$B_j$ may be found as follows. Let~$W_j^1=\{j\}$, and
    for~$k\ge1$, let
    \[ W_j^{k+1}=\{n\in\bN\colon \exists\,m\in W_j^k \text{
      s.t.~either $A_m\cap A_n\ne\emptyset$ or $B_m\cap
      B_n\ne\emptyset$}\}.\] Let~$W_j=\bigcup_{k\ge1} W_j^{k}$, and
    consider the Borel sets $R_j=\bigcup_{n\in W_j} A_n$ and
    $C_j=\bigcup_{n\in W_j}B_n$. By construction,~$G_j=G[R_j,C_j]$ is
    Borel. It is easy to see that~$G_j$ is  the component of~$G$
    containing~$A_j$ and~$B_j$, and that every component of~$G$ is
    of this form for some~$j$. Discard duplicates and relabel so that
    $G_j\ne G_k$ for $j\ne k$; the families~$\{R_j\}$ and $\{C_j\}$
    are then disjoint.  Extending each family to a countable Borel
    partition of~$X$ and applying Lemma~\ref{lem:direct-sums-cts}, we
    see that~$\|\Phi\|=\sup_j\|\Phi_j\|$
    where~$\Phi_j=\Gamma^{-1}([\chi_{G_j}])$.

  \item Let~$F$ be a countable induced subgraph of~$G$, so that
    $F=G[A,B]$ for countable sets~$A,B\subseteq X$.  Considering the
    functions $f|_A\in \ell^\infty(A,\ell^2)$ and $g|_B\in
    \ell^\infty(B,\ell^2)$, we see that $\|F\|<\eta$
    by~\cite[Corollary~8.8]{paulsen-book}.

  \item Now suppose
    that~$F=\df(G)=G[A,B]$. By~\cite[Corollary~8.8]{paulsen-book},
    there are functions $f_A\in \ell^\infty(A,\ell^2)$ and~$g_B\in
    \ell^\infty( B, \ell^2)$ so that~\[\langle f_A,g_B\rangle = \chi_{\df(G)}\colon A\times B\to \{0,1\}%
    \text{ and }\|f_A\|\,\|g_B\|=\|\df(G)\|.\] For~$x,y\in X$,
    write \[G_x=\{y\in X\colon (x,y)\in G\} \quad\text{and}\quad G^y=\{x\in X\colon
    (x,y)\in G\}.\] For each~$a\in A$ and~$b\in B$, the equivalence classes
    $S(a)=\{ x\in X\colon G_a=G_x\}$ and~$T(b)=\{y\in X\colon
    G^b=G^y\}$ are all Borel; indeed,
  \[ S(a)=f^{-1}\left( f(a)+\{g(y)\colon y\in
    Y\}^\perp\right)\]and\[
  T(b)=g^{-1}\left(g(b)+\{f(x)\colon x\in X\}^\perp\right).\] Hence
  $\tilde f=\sum_{a\in A} f_A(a)\chi_{S(a)}$ and $\tilde g=\sum_{b\in
    B} g_B(b)\chi_{T(b)}$ are Borel functions~$X\to\ell^2$, and
  $\chi_{G}(x,y)=\langle \tilde f(x),\tilde g(y)\rangle$ for
  every~$x,y\in X$. So \[\|\Phi\|\leq \|\tilde f\|\,\|\tilde g\| \leq
  \|f_A\|\,\|g_B\|=\|\df(G)\|.\qedhere\]
  \end{enumerate}
\end{proof}

\begin{cor} Let~$\H$ be a separable Hilbert space, and let~$\D$ be a
  masa in~$\B(\H)$.  The set~$\N(\D)=\{\|\Phi\|\colon \Phi\in \cbn,\
  \text{$\Phi$ idempotent}\}$ satisfies
  \[ \N(\D)\subseteq \{
  \eta_0,\eta_1,\eta_2,\eta_3,\eta_4,\eta_5\}\cup [\eta_6,\infty).\]
\end{cor}
\begin{proof}
  Let~$k\in \{1,2,3,4,5,6\}$ and suppose that~$\Phi\in \cbn$ is
  idempotent with~$\eta_k>\|\Phi\|$.  Taking $\eta=\eta_k$,
  let~$G,f,g,\Phi_j$ be as in
  Proposition~\ref{prop:biggraph}. Since~$\|\Phi\|=\sup_j\|\Phi_j\|$,
  every~$\Phi_j$ has~$\|\Phi_j\|<\eta_k$.  Hence we may assume
  that~$\Phi=\Phi_1$, so that $G$ is connected. Recall from
  \sectmark\ref{sec:bipartite} that~$\F(G)$ is the set of (isomorphism
  classes of) finite, connected, twin-free subgraphs of~$G$. If~$F\in
  \F(G)$, then~$\|F\|<\eta_k$ by Proposition~\ref{prop:biggraph}(3),
  so $\|F\|\leq \eta_{k-1}$ by Theorem~\ref{thm:gaps}. By
  Corollary~\ref{cor:characterisations}, $\F(G)$ consists entirely of
  induced subgraphs of some finite bipartite graph, so~$\F(G)$ is
  finite. By~Lemma~\ref{lem:finite-connected-subgraphs}, $\df(G)\in
  \F(G)$, so by Proposition~\ref{prop:biggraph}(4), $\|\Phi\|\leq
  \|\df(G)\|\leq \eta_{k-1}$.
\end{proof}
\begin{ques}
  Let~$\H$ be an infinite-dimensional separable Hilbert space. Do we
  have~$\N(\D)=\N$ for every masa~$\D$ in~$\B(\H)$?
\end{ques}

\section{Random Schur idempotents}\label{sec:random}

For~$0<p<1$ and~$m,n\in\bN$, let~$\G(m,n,p)$ be the probability
space of bipartite graphs in $\Gamma(m,n)$ where each of the possible
$mn$ edges appears independently with probability~$p$.

\begin{ques}
  How does $\bE_{m,n,p}(\|G\|)$, the expected value of the norm of the
  Schur idempotent arising from~$G\in \G(m,n,p)$, behave as a function of~$m$
  and~$n$?
\end{ques}

Here is a crude result in this general direction.
\begin{prop}
  If~$0<p<1$, then $\bE_{m,n,p}(\|G\|)\to \infty$ as $\min\{m,n\}\to \infty$.
\end{prop}
\begin{proof}
  Let~$s,t\in \bN$, fix~$H\in \Gamma(s,t)$ and let us write
  $\bP_{m,n,p}( H\leq G )$ for the probability that a random
  graph~$G\in \G(m,n,p)$ contains an induced subgraph isomorphic
  to~$H$. We claim that
  \[ \bP_{m,n,p}( H\leq G ) \to 1\quad\text{as } \min\{m,n\}\to
  \infty.\] Indeed, as in~\cite[Proposition~11.3.1]{diestel}, one can
  see that the complementary event $H\not\leq G$ satisfies
  \[ \bP_{m,n,p}(H\not\leq G) \leq
  (1-r)^{\min\{\lfloor{m/s}\rfloor,\lfloor{n/t}\rfloor\}}\to
  0\quad\text{as } \min\{m,n\}\to \infty\] where~$r>0$ is the
  probability that a random graph in~$\G(s,t,p)$ is isomorphic to~$H$.
  Hence
  \begin{align*} \bE_{m,n,p}(\|G\|)&=\sum_{G\in \Gamma(m,n)} \|G\|\,\bP_{m,n,p}(\{G\})\\
    &\geq \sum_{H\leq G\in \Gamma(m,n)} \|G\| \,\bP_{m,n,p}(\{G\})\\
    &\geq \|H\|\sum_{H\leq G\in \Gamma(m,n)}\,\bP_{m,n,p}(\{G\})=
    \|H\|\, \bP_{m,n,p}(H \leq G).
  \end{align*}
  By Proposition~\ref{prop:basic-graphs}(2), $\bE_{m,n,p}\|G\|$
  increases as~$\min\{m,n\}$ increases, so
  \[ \lim_{\min\{m,n\}\to \infty} \bE_{m,n,p}(\|G\|) \geq
  \sup\{ \|H\|\colon H\in \Gamma(s,t),\ s,t\in\bN\}=\infty.\qedhere\]
\end{proof}

For~$p=1/2$, we can say more about the growth rate
of~$\bE_{m,n,p}(\|G\|)$.  Doust~\cite{doust} shows that if~$1\leq
q<\infty$, then there is a constant~$K>0$ so that the norm~$\|G\|_q$
of a randomly chosen $(n,n)$ bipartite graph~$G$ acting as a Schur
multiplier on the Schatten $q$-class satisfies
\[\bE_{n,n,1/2}\|G\|_q\geq  K n^{|\frac1q-\frac12|}.\]
We are grateful to C\'edric Arhancet for pointing out Doust's work,
and for remarking that since $\|G\|=\|G\|_1$ by duality, this estimate
yields
\[ \bE_{n,n,1/2}\|G\|\geq K\sqrt n.\] We now show that we can replace
$K\sqrt n$ with $\frac 1{8\sqrt2} \sqrt{n}-1$.

\begin{lem}\label{lem:probs}
  Let~$m,n\in\bN$, fix an~$m\times n$ matrix~$A$ with complex entries
  and let~$\mu$ be the uniform probability measure
  on~$M_{m,n}(\{-1,1\})$.  If
  \[ \int_{\epsilon\in M_{m,n}(\{-1,1\})}
  \|\epsilon\schur A\|_\schur\,d\mu(\epsilon)=M,\] then
  \[\|\epsilon\schur A\|_\schur\leq 4M\] for
  every~$\epsilon\in M_{m,n}(\{-1,1\})$.
\end{lem}
\begin{proof}
  Let~$\nu$ be the probability measure on~$M_{m,n}(\bT)=\bT^{m\times
    n}$ which is the product of $m\times n$ copies of normalised Haar
  measure on~$\bT$.  The arguments in~\cite[\sectmark 2.6]{kahane}
  show that
  \[ \int_{z\in M_{m,n}(\bT)} \|\Re(z)\schur A\|_\schur \,d\nu(z) \leq
  M\]and\[\int_{z\in M_{m,n}(\bT)} \|\Im(z)\schur
  A\|_\schur \,d\nu(z) \leq M,\] where~$\Re(z)=[\Re(z_{ij})]$
  and~$\Im(z)=[\Im(z_{ij})]$. Hence \[\int_{z\in M_{m,n}(\bT)}
  \|z\schur A\|_\schur \,d\nu(z)\leq 2M.\] By~\cite[Theorem~2.2(i)
  and~Remark~2.3]{pisier}, $A=B+C$ where $c(B)\leq 2M$ and
  $c(C^t)\leq 2M$. For any~$\epsilon\in M_{m,n}(\{-1,1\})$, we
  have
  \[\|\epsilon\schur B\|_{\schur} \leq c(\epsilon\schur B)=c(B)\leq 2M\]
  and similarly~$\|\epsilon\schur C\|_{\schur}\leq 2M$, so
  \[\|\epsilon\schur A\|_\schur \leq  
  \|\epsilon\schur B\|_\schur+\|\epsilon\schur C\|_\schur \leq 4M.\qedhere\]
\end{proof}

\begin{prop}
  $\bE_{m,n,1/2}(\|G\|)\geq \frac1{8}\sqrt{\frac k2}-1$
  where~$k=\min\{m,n\}$.
\end{prop}
\begin{proof}
  \newcommand{\avG}{\bE_{m,n,1/2}(\|G\|)}
  Let~$\mu$ be the probability measure of the lemma, and write \[M=\int
  \|\epsilon\|_\schur\, d\mu(\epsilon).\] Note that
  \[ \avG = \int \|2\epsilon -\mathbf 1\|_{\schur}\,d\mu(\epsilon)
  \geq 2 M - 1,\] where $\mathbf 1$ is the all ones matrix. On the
  other hand,~\cite[Theorem~2.4]{dav-don} implies that there is a
  matrix~$\epsilon\in M_{m,n}(\{\pm1\})$ with
  $\|\epsilon\|_{\schur}\geq \frac14 \sqrt{\frac{mn}{m+n}}\geq
  \frac14\sqrt{\frac k2}$. By Lemma~\ref{lem:probs}, $M\geq \frac14
  \|\epsilon\|_{\schur}$. Combining these three inequalities gives the
  desired lower bound on~$\avG$.
\end{proof}

\section*{Acknowledgements}

The author is grateful to Chris Boyd, Ken Davidson, Charles Johnson
and Ivan Todorov for stimulating discussion and correspondence on
topics related to this paper.

\end{document}